\documentclass{birkjour_t2}
 \newtheorem{thm}{Theorem}[section]
 \newtheorem{cor}[thm]{Corollary}
 \newtheorem{lem}[thm]{Lemma}
 
 \theoremstyle{definition}
 \newtheorem{defn}[thm]{Definition}
 \theoremstyle{remark}
 
 \newtheorem{assumption}[thm]{Assumption}
 \newtheorem*{ex}{Example}
 \numberwithin{equation}{section}
\usepackage{color}

\begin{document}

%
%
%
%
%
%
%
%
%

\title[Discrete diffusive SIR epidemic model]
 {Traveling wave solutions for a class of discrete diffusive SIR epidemic model}

\author[Zhang]{Ran Zhang}

\address{%
School of Mathematics\\
Harbin Institute of Technology\\
Harbin 150001 Heilongjiang\\
People’s Republic of China}

\email{16B312002@hit.edu.cn}

\thanks{The authors were supported by Natural Science Foundation of China (11871179; 11771374).}
\author[Wang]{Jinliang Wang}
\address{%
School of Mathematical Sciences\\
Heilongjiang University\\
Harbin 150080 Heilongjiang\\
People's Republic of China}
\email{jinliangwang@hlju.edu.cn}
\author[Liu]{Shengqiang Liu}
\address{%
School of Mathematical Sciences\\
Tiangong University\\
Tianjin 300387 Tianjin\\
People's Republic of China}
\email{sqliu@tiangong.edu.cn}


\date{January 1, 2004}

\begin{abstract}
This paper is concerned with the conditions of existence and nonexistence of traveling wave solutions (TWS) for a class of discrete diffusive epidemic models. We find that the existence of TWS is determined by the so-called basic reproduction number and the critical wave speed: When the basic reproduction number $\Re_0>1$, there exists a critical wave speed $c^*>0$, such that for each $c \geq c^*$ the system admits a nontrivial TWS and for $c<c^*$ there exists no nontrivial TWS for the system. In addition, the boundary asymptotic behaviour of TWS is obtained by constructing a suitable Lyapunov functional and employing Lebesgue dominated convergence theorem. Finally, we apply our results to two discrete diffusive epidemic models to verify the existence and nonexistence of TWS.
\end{abstract}

\maketitle
\section{Introduction}\label{sec:Inc}
\def\d {{\rm d}}
In a pioneering work, the classical Susceptible-Infectious-Recovered (SIR) epidemic model was introduced by Kermack and McKendrick \cite{KermackMcKendrickPRLSA1927} in 1927. Since then,
epidemic modeling has became one of the most important tools to study  spread of the disease, we refer readers to a good survey \cite{Hethcote2000} on this topic. In order to understand the geographic spread of infectious disease, the spatial
effect would give insights into disease spread and control. Due to this fact, epidemic models with spatial diffusion have been studied for decades. Considering spatial effects, Hosono and Ilyas \cite{HosonoIlyasMMMAS1995} proposed and studied the following SIR epidemic model with diffusion:
\begin{equation}
\label{PreModel1}\left\{
\begin{array}{ll}
\vspace{2mm}
\displaystyle   \frac{\partial S(x,t)}{\partial t}= d_1 \Delta S(t,x) -  \beta S(x,t)I(x,t),&\ x\in \mathbb{R},\ t>0,\\
\displaystyle   \frac{\partial I(x,t)}{\partial t}= d_2 \Delta I(t,x) +  \beta S(x,t)I(x,t) - \gamma I(x,t),&\ x\in \mathbb{R},\ t>0,\\
\end{array}\right.
\end{equation}
with initial conditions
\[
S(x,0) = S_0(x),\ \ I(x,0) = I_0(x) >0,
\]
where $S(x,t)$ and $I(x,t)$ denote the densities of susceptible and infected individuals at position $x$ and time $t$, respectively; $d_i(i=1,2)$ are the diffusion rates of each compartments; $\beta$ denotes the transmission rate between susceptible and infected individuals; $\gamma$ is the remove rate. All parameters in system (\ref{PreModel1}) are assumed to be positive. The authors proved the existence of traveling wave solutions of system (\ref{PreModel1}) with a constant speed when $d_2 = 1$.
In recent years,  many researchers have paid attention to study the traveling wave solutions for diffusive epidemic models (see, for example, \cite{BaiZhangCNSNS2015,DucrotMagalNon2011,FuJMAA2016,HeTsaiZAMP2019,LamWangZhangSIAMJMA2018,LiLiYangAMC2014,LiXuZhangDCDS2017,TianYuanSCM2017,WangMaDCDSB2018,ZhaoWangRuanNonlinearity2017,ZhangXuAMC2013} and references therein).

However, there are relatively few works on epidemic models with discrete spatial structure. In contrast to continuous media, lattice dynamical systems is more realistic in describing the discrete diffusion (for example, patch environment \cite{SanWangJMAA2019}). Lattice dynamical systems are systems of ordinary differential equations with a discrete spatial structure.
Such systems arise from practical backgrounds, such as biology \cite{WengHuangWuIMAJAM2003,FangWeiZhaoPRSAMPES2010,WuWengRuanEJAM2015,YangZhangSCM2018,Han2019},  chemical reaction \cite{KapralJMC1991,ErneuxNicolisPD1993} and material science \cite{BatesChmajARMA1999,HallareVleckSIAMADS2011}.
In a recent paper \cite{FuGuoWuJNCA2016}, Fu et al. studied the existence of traveling wave solutions for a lattice dynamical system arising in a discrete diffusive epidemic model:
\begin{equation}
\label{PreModel2}\left\{
\begin{array}{l}
\vspace{2mm}
\displaystyle   \frac{\d S_n(t)}{\d t}= [S_{n+1}(t) + S_{n-1}(t) - 2S_n(t)] - \beta S_n(t)I_n(t),\ \ n\in\mathbb{Z},\\
\displaystyle   \frac{\d I_n(t)}{\d t}= d[I_{n+1}(t) + I_{n-1}(t) - 2I_n(t)] + \beta S_n(t)I_n(t)  - \gamma I_n(t),\ \ n\in\mathbb{Z},\\
\end{array}\right.
\end{equation}
where $S_n(t)$ and $I_n(t)$ denote the populations densities of susceptible and infectious individuals at niche $n$ and time $t$, respectively; $1$ and $d$ denote the random migration coefficients for susceptible and infectious individuals, respectively; $\beta$ is the transmission coefficient between susceptible and infectious individuals; $\gamma$ is the recovery rate of infectious individuals.
Note that system \eqref{PreModel2} is a spatially discrete version of system \eqref{PreModel1}. It was proved in \cite{FuGuoWuJNCA2016} that the conditions of existence and nonexistence of traveling wave solution for system \eqref{PreModel2} are determined by a threshold number and the critical wave speed $c^*$. If the threshold number is greater than 1, then there exists a traveling wave solution for any $c>c^*$ and there is no traveling wave solutions for $c<c^*$. Also, the non-existence of traveling wave solutions for the threshold number is less than 1 was derived. Furthermore, Wu \cite{WuJDE2017} studied the existence of traveling wave solutions with critical speed $c=c^*$ of system \eqref{PreModel2}. In \cite{ZhangWuIJB2019} and \cite{ZhouSongWeiJDE2019}, two models with saturated incidence rate are considered, and they also investigated the existence and nonexistence of traveling wave solutions. By introducing the constant recruitment, Chen et al. \cite{ChenGuoHamelNon2017} studied the traveling wave solutions for the following discrete diffusive epidemic model:
\begin{equation}
\label{PreModel}\left\{
\begin{array}{l}
\vspace{2mm}
\displaystyle   \frac{\d S_n(t)}{\d t}= [S_{n+1}(t) + S_{n-1}(t) - 2S_n(t)] + \mu -  \beta S_n(t)I_n(t) - \mu S_n(t),\ \ n\in\mathbb{Z},\\
\displaystyle   \frac{\d I_n(t)}{\d t}= d[I_{n+1}(t) + I_{n-1}(t) - 2I_n(t)] + \beta S_n(t)I_n(t)  - (\gamma+\mu) I_n(t),\ \ n\in\mathbb{Z},\\
\end{array}\right.
\end{equation}
where $\mu$ is the input rate of the susceptible population, meanwhile, the death rates of susceptible and infectious individuals are also assumed to be $\mu$. In \cite{ChenGuoHamelNon2017}, the authors showed that the existence of traveling wave solutions connecting the disease-free equilibrium to the endemic equilibrium, but it still remains open whether the traveling wave solutions converge to the endemic equilibrium at $+\infty$.
As explained in \cite{ChenGuoHamelNon2017}, the main difficulties come from the fact that (\ref{PreModel})  is a system and is non-local.
In fact, the traveling wave solutions of \eqref{PreModel2} and \eqref{PreModel} are totally different: The disease will always die out for the system like \eqref{PreModel2} without constant recruitment, that is, $I$ tends to 0 as $\xi\rightarrow\pm\infty$, where $\xi=n+ct$ is the wave profile to be introduced in the next section; However, for the diffusive model with positive constant recruitment, it is more likely to get that $I(\xi)\rightarrow0$ as $\xi\rightarrow-\infty$ and $I(\xi)\rightarrow I^*$ as $\xi\rightarrow+\infty$, where $I^*$ is the positive endemic equilibrium (see \cite{LiLiYangAMC2014} for nonlocal diffusive epidemic model; \cite{FuJMAA2016} for random diffusive epidemic model). Therefore, it naturally   raises a question:   For discrete diffusive systems, does the traveling wave solutions converge to the endemic equilibrium as $\xi\rightarrow+\infty$? This constitutes our first motivation of the present paper.

Our second motivation is the nonlinear incidence rate which plays a critical role in the epidemic modeling~\cite{AndersonMay1991}, for the discrete diffusive systems with nonlinear incidence rate, will the traveling wave solutions still converge to the endemic equilibrium
as  $\xi\rightarrow+\infty$? Traditionally, the incidence rate of an infectious disease in most of the literature is assumed to be of mass action form $\beta SI$~\cite{AndersonMay1991}. Yet  the disease transmission process is generally unknown~\cite{KorobeinikovMainiMMB2005}, some nonlinear incidence rates  have been introduced and studied, for example, the saturated incidence rate with $f(I) = \frac{I}{1+\alpha I}$ by \cite{CapassoSerioMB1978}, the saturated nonlinear incidence rate with $f(I) = \frac{I}{1+\alpha I^p}(0<p<1)$ by \cite{LiuLevinIwasaJMB1986}, and so on. For more general cases, Capasso et al.~\cite{CapassoSerioMB1978} considered a more general incidence rate with the form $Sf(I)$. It is seems that the general nonlinear incidence rate could bring nontrivial challenges in analysis. Therefore, it is of great significance to study the convergence property of traveling wave solutions of the system with nonlinear incidence rate.

In this paper, we consider a discrete diffusive SIR epidemic model with general nonlinear incidence rate. The main model of this paper is formulated as the following system:
\begin{equation}
\label{PreModel3}\left\{
\begin{array}{l}
\vspace{2mm}
\displaystyle   \frac{\d S_n(t)}{\d t}= d_1[S_{n+1}(t) + S_{n-1}(t) - 2S_n(t)] + \Lambda - \beta S_n(t)f(I_n(t)) - \mu_1 S_n(t),\ \ n \in \mathbb{Z},\\
\vspace{2mm}
\displaystyle   \frac{\d I_n(t)}{\d t}= d_2[I_{n+1}(t) + I_{n-1}(t) - 2I_n(t)] + \beta S_n(t)f(I_n(t))  - \gamma I_n(t) - \mu_2 I_n(t),\ \ n \in \mathbb{Z},\\
\displaystyle   \frac{\d R_n(t)}{\d t}= d_3[R_{n+1}(t) + R_{n-1}(t) - 2R_n(t)] + \gamma I_n(t) - \mu_1 R_n(t),\ \ n \in \mathbb{Z},
\end{array}\right.
\end{equation}
where $S_n(t)$, $I_n(t)$ and $R_n(t)$ denote the densities of susceptible, infectious and removed individuals at niche $n$ and time $t$, respectively;
$d_i(i=1,2,3)$ is the random migration coefficients for each compartments; $\Lambda$ is the input rate of susceptible individuals.
The biological meaning of other parameters are the same as in model \eqref{PreModel}.
This paper aims to study the existence and convergence property of traveling wave solutions of model (\ref{PreModel3}), and one of our results will answer the open problem proposed in \cite{ChenGuoHamelNon2017}.

Since $R_n(t)$ is decoupled from other equations and denote $\mu_2 = \gamma + \mu_1$, then we only need to study the following system:
\begin{equation}
\label{Model}\left\{
\begin{array}{l}
\vspace{2mm}
\displaystyle   \frac{\d S_n(t)}{\d t}= d_1[S_{n+1}(t) + S_{n-1}(t) - 2S_n(t)] + \Lambda - \beta S_n(t)f(I_n(t)) - \mu_1 S_n(t),\ \ n \in \mathbb{Z},\\
\displaystyle   \frac{\d I_n(t)}{\d t}= d_2[I_{n+1}(t) + I_{n-1}(t) - 2I_n(t)] + \beta S_n(t)f(I_n(t))  - \mu_2 I_n(t),\ \ n \in \mathbb{Z}.
\end{array}\right.
\end{equation}
We make the following assumptions on function $f$.

\begin{assumption}\label{Assum}
Assume the function $f$ satisfying
\begin{description}
  \item[(A1)] $f(I)\geq 0$ and $f'(I) > 0$ for all $I\geq0$, $f(I) = 0$ if and only if $I = 0$.
  \item[(A2)] $\frac{f(I)}{I}$ is continuous and monotonously non-increasing for all $I\geq0$ and $\lim\limits_{I\rightarrow 0^+}\frac{f(I)}{I}$ exists.
\end{description}
\end{assumption}

 The  conditions of  Assumption \ref{Assum} are   satisfied in all the following specific incidence rates:
\begin{enumerate}
  \item[(i)] the bilinear incidence rate with $f(I)=I$ (see \cite{AndersonMay1991});
  \item[(ii)] the saturated incidence rate with $f(I) = \frac{I}{1+\alpha I}$ (see \cite{CapassoSerioMB1978});
  \item[(iii)] the saturated nonlinear incidence rate with $f(I) = \frac{I}{1+\alpha I^p}$, where $\alpha>0$ and $0<p<1$ (a special case in \cite{LiuLevinIwasaJMB1986}, see also \cite{MuroyaKuniyaEnatsuDCDSB2015});
  \item[(iv)]  the nonlinear incidence rate with $f(I) = \frac{I}{1+kI+\sqrt{1+2kI}}$ (see \cite{HeesterbeekMetzJMB1993,ThiemeJDE2011});
  \item[(v)]  the nonlinear incidence rate with $f(I) = \frac{I}{(\epsilon^\alpha + I^\alpha)^\gamma}$, where $\epsilon, \alpha, \gamma>0$ and $\alpha\gamma<1$  (see \cite{ThiemeJDE2011});
  \item[(vi)] the nonlinear incidence rate for pathogen transmission in infection of insects with $f(I) = k \ln \left(1+\frac{\nu I}{k}\right)$, which could be described by epidemic model (see \cite{BriggsGodfray1995}).
\end{enumerate}

Hence, system (\ref{Model}) covers many models as special cases. Now, we introduce some results on the system  (\ref{Model})  without migration, which takes the form as:
\begin{equation}
\label{ODEModel}\left\{
\begin{array}{l}
\vspace{2mm}
\displaystyle   \frac{\d S(t)}{\d t}= \Lambda -  \beta S(t)f(I(t)) - \mu_1 S(t),\\
\displaystyle   \frac{\d I(t)}{\d t}= \beta S(t)f(I(t)) - \mu_2 I(t).
\end{array}\right.
\end{equation}
It is well-known that the global dynamics of \eqref{ODEModel} is completely determined by the basic reproduction number $\Re_{0} = \frac{\beta S_0 f'(0)}{\mu_2}$ (see \cite{KorobeinikovBMB2006}): that is, if the number is less than unity, then the disease-free equilibrium $E_0=(S_0,0)=(\Lambda/\mu_1,0)$ is globally asymptotically stable, while if the number is greater than unity, then a positive endemic equilibrium $E^*=(S^*,I^*)$ exists and it is globally asymptotically stable, where $E^*$ satisfy
\begin{equation}
\label{EStar}\left\{
\begin{array}{l}
\vspace{2mm}
\displaystyle   \Lambda - \beta S^*f(I^*) - \mu_1 S^* = 0,\\
\displaystyle   \beta S^*f(I^*) - \mu_2 I^* = 0.
\end{array}\right.
\end{equation}

The organization of this paper is as follows. In Section 2, we apply Schauder's fixed point theorem to construct a family of solutions of the truncated problem. In Section 3, we show the existence and boundedness of traveling wave solutions. Further, we use a Lyapunov functional to show that the convergence of traveling wave solutions at $+\infty$. In Section 4, we investigate the nonexistence of traveling wave solutions by using two-sided Laplace transform. At last, there is an application for our general results and a brief discussion.

\section{Preliminaries}
In this section, since system \eqref{Model} does not enjoy the comparison principle, we will construct a pair upper and lower solutions and apply Schauder's fixed point theorem
to investigate the existence of traveling wave solutions of system (\ref{Model}).
Consider traveling wave solutions which can be expressed as bounded profiles of continuous variable $n+ct$ such that
\begin{equation}\label{c}
S_n(t) = S(n+ct)\ \ \textrm{and}\ \ I_n(t) =  I(n+ct).
\end{equation}
where $c$ denotes the wave speed. Let $\xi=n+ct$, then we can rewrite system (\ref{Model}) as follows:
\begin{equation}
\label{WaveEqu}\left\{
\begin{array}{l}
\vspace{2mm}
\displaystyle   cS'(\xi)= d_1J[S](\xi) + \Lambda - \mu_1 S(\xi) - \beta S(\xi)f(I(\xi)),\\
\displaystyle   cI'(\xi)= d_2J[I](\xi) + \beta S(\xi)f(I(\xi)) - \mu_2 I(\xi)
\end{array}\right.
\end{equation}
for all $\xi\in \mathbb{R}$, where $J[\phi](\xi):=  \phi(\xi+1) +  \phi(\xi-1) - 2 \phi(\xi)$.
We want to find traveling wave solutions with the following asymptotic boundary conditions:
\begin{equation}\label{Bound1}
\lim_{\xi\rightarrow-\infty}(S(\xi), I(\xi))=(S_0, 0),
\end{equation}
and
\begin{equation}\label{Bound2}
\lim_{\xi\rightarrow+\infty}(S(\xi), I(\xi))=(S^*, I^*).
\end{equation}
where $(S_0,0)$ is the disease-free equilibrium and $(S^*,I^*)$ is the positive endemic equilibrium, which have been defined in Section 1. Linearizing the second equation of system (\ref{WaveEqu}) at disease-free equilibrium $(S_0, 0)$, we have
\begin{equation}
\label{LinearModel}
c I'(\xi)= d_2J[ I](\xi) - \mu_2 I(\xi) + \beta S_0 f'(0) I(\xi).
\end{equation}
Letting $I(\xi)=e^{\lambda \xi}$ and substituting it into (\ref{LinearModel}), yields
\[
d_2[e^\lambda + e^{-\lambda} - 2] - c\lambda + \beta S_0 f'(0) - \mu_2 = 0.
\]
Denote
\begin{equation}\label{Delta}
\Delta(\lambda,c) = d_2[e^\lambda + e^{-\lambda} - 2] - c\lambda + \beta S_0 f'(0) - \mu_2.
\end{equation}
By some calculations, we have
\begin{align*}
&\Delta(0,c) = \beta S_0 f'(0) - \mu_2,\ \ \ \lim_{c\rightarrow+\infty} \Delta (\lambda,c) = -\infty,\\
&\frac{\partial \Delta(\lambda, c)}{\partial\lambda} = d_2[e^\lambda - e^{-\lambda}] - c,\ \ \ \frac{\partial \Delta(\lambda, c)}{\partial c}= -\lambda < 0,\\
&\frac{\partial^2 \Delta(\lambda, c)}{\partial\lambda^2} = d_2[e^\lambda + e^{-\lambda}] > 0,\ \ \ \frac{\partial \Delta(\lambda, c)}{\partial\lambda}\bigg|_{(0,c)} = -c < 0,
\end{align*}
for $\lambda>0$ and $c>0$. Therefore, we have the following lemma.
\begin{lem}\label{WaveSpeed}
Let $\Re_{0}>1.$ There exist $c^*>0$ and $\lambda^*>0$ such that
\[
\frac{\partial \Delta(\lambda, c)}{\partial \lambda}\bigg|_{(\lambda^*,c^*)} = 0\ \ \textrm{and}\ \ \Delta(\lambda^*,c^*) = 0.
\]
Furthermore,
\begin{description}
  \item[(i)] if $c=c^*,$ then $\Delta(\lambda, c)=0$ has only one positive real root $\lambda^*;$
  \item[(ii)] if $0<c<c^*,$ then $\Delta(\lambda, c)>0$ for all $\lambda\in(0, \lambda_M),$ where  $\lambda_M\in(0,+\infty]$;
  \item[(iii)] if $c>c^*,$ then $\Delta(\lambda, c)=0$ has two positive real roots $\lambda_1,$ $\lambda_2$ with $\lambda_1<\lambda^*<\lambda_2$.
\end{description}
\end{lem}
From Lemma \ref{WaveSpeed}, we have
\begin{equation}
\label{WaveSpeedRel}
\Delta(\lambda,c)\left\{
\begin{array}{l}
\vspace{2mm}
\displaystyle   >0\ \ \ {\rm for}\ \ \ \lambda<\lambda_1,\\
\vspace{2mm}
\displaystyle   <0\ \ \ {\rm for}\ \ \ \lambda_1<\lambda<\lambda_2,\\
\displaystyle   >0\ \ \ {\rm for}\ \ \ \lambda>\lambda_2.
\end{array}\right.
\end{equation}
In the following of this section, we always fix $c>c^*$ and $\Re_0>1$.

\subsection{Construction of upper and lower solutions}
\begin{defn}\label{Def}
$(S^+(\xi), I^+(\xi))$ and $(S^-(\xi), I^-(\xi))$ are called a pair upper and lower solutions of \eqref{WaveEqu} if $S^+, I^+, S^-, I^-$ satisfy
\[
\left\{
\begin{array}{l}
\vspace{2mm}
\displaystyle   d_1J[S^+](\xi) - c{S^+}'(\xi) + \Lambda - \mu_1 S^+(\xi) - \beta S^+(\xi)f(I^-(\xi)) \leq 0,\\
\vspace{2mm}
\displaystyle   d_2J[I^+](\xi) - c{I^+}'(\xi) + \beta S^+(\xi)f(I^+(\xi)) - \mu_2 I^+(\xi) \leq 0,\\
\vspace{2mm}
\displaystyle   d_1J[S^-](\xi) - c{S^-}'(\xi) + \Lambda - \mu_1 S^-(\xi) - \beta S^-(\xi)f(I^+(\xi)) \geq 0,\\
\displaystyle   d_2J[I^-](\xi) - c{I^-}'(\xi) + \beta S^-(\xi)f(I^-(\xi)) - \mu_2 I^-(\xi) \geq 0.\\
\end{array}\right.
\]
\end{defn}
Define the following functions:
\begin{equation}
\label{UpLowSolution}\left\{
\begin{array}{l}
\vspace{2mm}
\displaystyle   S^+(\xi)=S_0,\\
\displaystyle   I^+(\xi) = e^{\lambda_1 \xi},
\end{array}\right.
\ \
\left\{
\begin{array}{l}
\vspace{2mm}
\displaystyle   S^-(\xi)=\max\{S_0(1-M_1 e^{\varepsilon_1 \xi}),0\},\\
\displaystyle   I^-(\xi)=\max\{e^{\lambda_1\xi}(1-M_2e^{\varepsilon_2 \xi}),0\},
\end{array}\right.
\end{equation}
where $M_i$ and $\varepsilon_i(i=1,2)$ are some positive constants to be determined in the following lemmas.
Now we show that \eqref{UpLowSolution} are a pair upper and lower solutions of \eqref{WaveEqu}.
\begin{lem}\label{IUp}
The function $ I^+(\xi) = e^{\lambda_1 \xi}$ satisfies
\begin{equation}
\label{IUpEqu}
c {I^+}'(\xi)= d_2J[I^+](\xi) - \mu_2  I^+(\xi) + \beta S_0 f'(0) I^+(\xi).
\end{equation}
\end{lem}

\begin{lem}\label{SVUp}
The function $S^+(\xi)=S_0$ satisfies
\begin{equation}
\label{SVUpEqu}
c{S^+}'(\xi) \geq d_1J[S^+](\xi) + \Lambda - \mu_1 S^+(\xi) - \beta S^+(\xi)f(I^-(\xi)).
\end{equation}
\end{lem}
The proof of the above two lemmas are straightforward, so we omit the details.

\begin{lem}\label{LemLowS}
For each $0<\varepsilon_1<\lambda_1$ sufficiently small and $M_1>0$ large enough, the function $S^-(\xi)=\max\{S_0(1-M_1 e^{\varepsilon_1 \xi}),0\}$ satisfies
\begin{equation}\label{SLowEqu}
c{S^-}'(\xi) \leq d_1J[S^-](\xi) + \Lambda - \mu_1 S^-(\xi) - \beta S^-(\xi)f(I^+(\xi))
\end{equation}
with $\xi\neq\frac{1}{\varepsilon_1}\ln\frac{1}{M_1}:=\mathfrak{X}_1$.
\end{lem}
\begin{proof}
If $\xi>\mathfrak{X}_1$, then inequality (\ref{SLowEqu}) holds since $S^-(\xi)=S^-(\xi+1)=0$ and $S^-(\xi-1)\geq0$.
If $\xi<\mathfrak{X}_1$, then
\[
S^-(\xi)=S_0(1-M_1 e^{\varepsilon_1 \xi}),\ \ S^-(\xi-1)=S_0(1-M_1 e^{\varepsilon_1 (\xi-1)})\ \ \textrm{and}\ \ S^-(\xi+1)\geq S_0(1-M_1 e^{\varepsilon_1 (\xi+1)}).
\]
From the concavity of function $f( I)$, we have
\[
f( I^+(\xi)) \leq f'(0)  I^+(\xi),
\]
thus
\begin{align*}
&\ d_1J[S^-](\xi) + \Lambda - \mu S^-(\xi) - \beta S^-(\xi)f(I^+(\xi)) - c{S^-}'(\xi)\\
\geq &\ e^{\varepsilon_1 \xi} S_0\left[-M_1 (d_1e^{\varepsilon_1} + d_1e^{-\varepsilon_1} - 2d_1 - \mu - c \varepsilon_1) - \beta_1 f\left(e^{\lambda_1\xi}\right) e^{-\varepsilon_1\xi} + \beta_1 M_1 \varepsilon_1 f\left(e^{\lambda_1\xi}\right)\right]\\
\geq &\ e^{\varepsilon_1 \xi} S_0 \left[-M_1 (d_1e^{\varepsilon_1} + d_1e^{-\varepsilon_1} - 2d_1 - \mu - c \varepsilon_1) -\beta_1 f'(0) e^{\lambda_1\xi}e^{-\varepsilon_1\xi}\right].
\end{align*}
Select $0<\varepsilon_1<\lambda_1$ small enough such that $-d_1(2 - e^{\varepsilon_1} - e^{-\varepsilon_1}) - \mu - c \varepsilon_1 < 0$ and note that $e^{(\lambda_1 - \varepsilon_1)\xi}\leq 1$ since $\xi<\mathfrak{X}_1<0$. Hence, we need to choose
\[
M_1 \geq \frac{\beta_1 f'(0) }{d_1(2 - e^{\varepsilon_1} - e^{-\varepsilon_1}) + \mu + c \varepsilon_1}
\]
large enough to make sure inequality (\ref{SLowEqu}) holds. This completes the proof.
\end{proof}

\begin{lem}\label{LemLowI}
For each $\varepsilon_2>0$ sufficiently small and $M_2>\frac{\varepsilon_2}{\varepsilon_1}M_1$ large enough, the function $ I^-(\xi)=\max\{e^{\lambda_1\xi}(1-M_2e^{\varepsilon_2 \xi}),0\}$ satisfies
\begin{equation}\label{lowIEqu}
c I'(\xi) \leq d_2J[I](\xi) + \beta S^-(\xi)f(I(\xi)) - \mu_2 I(\xi)
\end{equation}
with $\xi \neq \frac{1}{\varepsilon_2}\ln \frac{1}{M_2}:=\mathfrak{X}_2$.
\end{lem}
\begin{proof}
If $\xi > \mathfrak{X}_2$, then inequality (\ref{lowIEqu}) holds since $I^-(\xi)=I^-(\xi-1)=0$ and $I^-(\xi+1)\geq 0$. If $\xi < \mathfrak{X}_2$, then
\[
I^-(\xi)=e^{\lambda_1\xi}(1-M_2e^{\varepsilon_2 \xi}),\ \ I^-(\xi-1)=e^{\lambda_1(\xi-1)}(1-M_2e^{\varepsilon_2 (\xi-1)}),\ \ I^-(\xi+1)\geq e^{\lambda_1(\xi+1)}(1-M_2e^{\varepsilon_2 (\xi+1)}),
\]
and inequality (\ref{lowIEqu}) is equivalent to the following inequality:
\begin{align}\label{I1}
& \ \beta S_0 f'(0) I^-(\xi) - \beta S^-(\xi) f (I^-(\xi))\\ \nonumber
\leq & \ d_2J[ I^-](\xi) - \mu_2  I^-(\xi) - c{ I^-}'(\xi) + \beta S_0 f'(0) I^-(\xi).
\end{align}
Note that $\frac{f(I)}{I}$ is non-increasing on $(0,\infty)$ in Assumption \ref{Assum}. For any $\epsilon\in(0,f'(0))$, then there exists a small positive number $\delta>0$ such that
\[
\frac{f( I)}{ I} \geq f'(0) - \epsilon,\ \ 0< I<\delta.
\]
For $0< I<\delta$, we have
\begin{align}\label{Equ1}
\nonumber \beta S_0 f'(0) I^-(\xi) - \beta S^-(\xi) f (I^-(\xi)) = & \left(\beta S_0 - \beta S^-(\xi) \frac{f(I^-(\xi))}{I^-(\xi)}\right) I^-(\xi)\\
\nonumber\leq & \left(\frac{\beta S_0  - \beta S^-(\xi)\frac{f(I^-(\xi))}{I^-(\xi)} + I^-(\xi)}{2}\right)^2\\
\leq & \left(\beta S_0 - \beta S^-(\xi) (f'(0) - \epsilon) +  I^-(\xi)\right)^2.
\end{align}
Recall that $\xi<\mathfrak{X}_2$, we can choose $M_2$ large enough such that
\[
0<I^-(\xi)<\delta\ \ {\rm and}\ \  S^-(\xi)\rightarrow S_0.
\]
Since (\ref{Equ1}) is valid for any $\epsilon$, we have
\[
\beta S_0 f'(0) I^-(\xi) - \beta S^-(\xi) f (I^-(\xi)) \leq [ I^-(\xi)]^2.
\]
Furthermore, the right-hand of (\ref{I1}) is satisfy
\[
d_2J[ I^-](\xi) - \mu_2  I^-(\xi) - c{ I^-}'(\xi) + \beta S_0 f'(0) I^-(\xi) \geq e^{\lambda_1\xi}\Delta(\lambda_1,c) - M_2 e^{\lambda_1+\varepsilon_2}\Delta(\lambda_1+\varepsilon_2,c).
\]
Using the definition of $\Delta(\lambda,c)$ and $[ I^-(\xi)]^2\leq e^{2\lambda_1\xi}$, noticing that $\Delta(\lambda_1+\varepsilon_2,c)<0$ for small $\varepsilon_2>0$ by (\ref{WaveSpeedRel}), then it suffices to show that
\[
e^{(\lambda_1 - \varepsilon_2)\xi} \leq - M_2 \Delta(\lambda_1+\varepsilon_2,c).
\]
The above inequality holds for $M_2$ large enough, since the left-hand side vanishes and the right-hand side tends to infinity as $M_2\rightarrow+\infty$.
This ends the proof.
\end{proof}

Hence, functions \eqref{UpLowSolution} are a pair upper and lower solutions of \eqref{WaveEqu} by the Definition \ref{Def}.

\subsection{Truncated problem}
Let $X>-\mathfrak{X}_2>0$. Define the following set
\begin{equation*}
\Gamma_X := \left\{(\phi, \psi)\in C([-X,X],\mathbb{R}^2)\left|
\begin{array}{l}
\vspace{2mm}
\displaystyle   S^-(\xi)\leq \phi(\xi) \leq S^+(\xi),\  I^-(\xi)\leq \psi(\xi) \leq I^+(\xi)\ \ \forall\xi\in[-X,X],\\
\displaystyle   \phi(-X)=S^-(-X),\ \ \psi(-X)=I^-(-X).
\end{array}\right.\right\}.
\end{equation*}
It is clear that $\Gamma_X$ is a nonempty bounded closed convex set in $C([-X,X],\mathbb{R}^2)$. For any $(\phi,\psi)\in C([-X,X],\mathbb{R}^2)$,
extend it as
\begin{equation*}
\label{hat1}
\hat{\phi}(\xi)=\left\{
\begin{array}{ll}
\displaystyle   \phi(X), &\mbox{for $\xi>X$,}
\\
\displaystyle   \phi(\xi), &\mbox{for $\xi\in[-X, X]$,}
\\
\displaystyle   S^-(\xi), &\mbox{for $\xi< -X$,}\\
\end{array}\right.\ \ \
\hat{\psi}(\xi)=\left\{
\begin{array}{ll}
\displaystyle   \psi(X), &\mbox{for $\xi>X$,}
\\
\displaystyle   \psi(\xi), &\mbox{for $\xi\in[-X, X]$,}
\\
\displaystyle   I^-(\xi), &\mbox{for $\xi< -X$.}\\
\end{array}\right.
\end{equation*}
Consider the following truncated initial problem:
\begin{equation}
\label{TruPro}\left\{
\begin{array}{l}
\vspace{2mm}
\displaystyle   cS'(\xi) + (2d_1+\mu_1+\alpha)S(\xi) = d_1\hat{\phi}(\xi+1) + d_1\hat{\phi}(\xi-1) + \Lambda + \alpha \phi(\xi) - \beta \phi(\xi)f(\psi(\xi)) := H_1(\phi,\psi),\\
\vspace{2mm}
\displaystyle   cI'(\xi) + (2d_2+\mu_2)I(\xi) = d_2\hat{\psi}(\xi+1) + d_2\hat{\psi}(\xi-1) + \beta \phi(\xi) f(\psi(\xi)) := H_2(\phi,\psi),\\
\displaystyle   (S,I)(-X) = (S^-,I^-)(-X),
\end{array}\right.
\end{equation}
where $(\phi,\psi)\in \Gamma_X$ and $\alpha$ is a constant large enough such that $H_1(\phi,\psi)$ is non-decreasing on $\phi(\xi)$. By the ordinary differential equation theory, system (\ref{TruPro}) has a unique solution $(S_X(\xi),I_X(\xi))$
satisfying $(S_X(\xi),I_X(\xi))\in C^1([-X,X],\mathbb{R}^2)$. Then we define an operator
\[
\mathcal{A} = (\mathcal{A}_1,\mathcal{A}_2):\Gamma_X\rightarrow C^1\left([-X,X],\mathbb{R}^2\right)
\]
by
\[
S_X(\xi)=\mathcal{A}_1(\phi,\psi)(\xi)\ \ {\rm and}\ \ I_X(\xi)=\mathcal{A}_2(\phi,\psi)(\xi).
\]

Next we show the operator $\mathcal{A}=(\mathcal{A}_1,\mathcal{A}_2)$ has a fixed point in $\Gamma_X$ by Schauder's fixed point theorem (see \cite[Corollary 2.3.10]{Chang2005}).

\begin{lem}
The operator $\mathcal{A}=(\mathcal{A}_1,\mathcal{A}_2)$ maps $\Gamma_X$ into itself.
\end{lem}
\begin{proof}
Firstly, we show that $S^-(\xi)\leq S_X(\xi)$ for any $\xi\in[-X,X].$
If $\xi\in(\mathfrak{X}_1, X),$ it follows that $S^-(\xi)=0$ and is a lower solution of the first equation of (\ref{TruPro}).
If $\xi\in(-X,\mathfrak{X}_1),$ then $S^-(\xi)=S_0(1-M_1 e^{\varepsilon_1 \xi})$, from Lemma \ref{LemLowS}, we have
\begin{align*}
&c{S^-}'(\xi) + (2d_1+\mu_1+\alpha)S^-(\xi) - d_1\hat{\phi}(\xi+1) -d_1 \hat{\phi}(\xi-1) - \Lambda - \alpha\phi(\xi) + \beta \phi(\xi)f(\psi(\xi))\\
\leq & c{S^-}'(\xi) - d_1J[S^-](\xi) - \Lambda + \mu_1 S^-(\xi) + \beta S^-(\xi)f( I^+(\xi))\\
\leq & 0,
\end{align*}
which implies that $S^-(\xi)=S_0(1-M_1 e^{\varepsilon_1 \xi})$ is a lower solution of the first equation of (\ref{TruPro}).
Thus $S^-(\xi)\leq S_X(\xi)$ for any $\xi\in[-X,X].$

Secondly, we show that $S_X(\xi)\leq S^+(\xi) = S_0$ for any $\xi\in[-X,X].$ In fact,
\begin{align*}
&c{S^+}'(\xi) + (2d_1+\mu_1+\alpha)S^+(\xi) - d_1\hat{\phi}(\xi+1) - d_1 \hat{\phi}(\xi-1) - \Lambda -\alpha\phi(\xi)+ \beta \phi(\xi)f(\psi(\xi))\\
\geq & \beta S_0f(I^-(\xi))\\
\geq & 0,
\end{align*}
thus $S^+(\xi)=S_0$ is an upper solution to the first equation of (\ref{TruPro}), which gives us $S_X(\xi)\leq S_0$ for any $\xi\in[-X,X].$

Similarly, we can show that
$I^-(\xi)\leq I_X(\xi)\leq I^+(\xi)$ for any $\xi\in[-X,X].$
This completes the proof.
\end{proof}

\begin{lem}
The operator $\mathcal{A}$ is completely continuous.
\end{lem}
\begin{proof}
Suppose $(\phi_i(\xi),\psi_i(\xi))\in\Gamma_X,\ i=1,2.$ Denote
\begin{align*}
S_{X,i}(\xi)=\mathcal{A}_1(\phi_i(\xi),\psi_i(\xi))\ \ {\rm and}\ \ I_{X,i}(\xi)=\mathcal{A}_2(\phi_i(\xi),\psi_i(\xi)),
\end{align*}
We show that the operator $\mathcal{A}$ is continuous. By direct calculation, we have
\[
S_X(\xi) = S^-(-X) e ^{-\frac{2d_1+\mu_1+\alpha}{c}(\xi+X)} + \frac{1}{c}\int_{-X}^\xi e ^{\frac{2d_1+\mu_1+\alpha}{c}(\tau+X)}H_1(\phi,\psi)(\tau)\d \tau,
\]
and
\[
I_X(\xi) = I^-(-X) e ^{-\frac{2d_2+\mu_2}{c}(\xi+X)} + \frac{1}{c}\int_{-X}^\xi e ^{\frac{2d_2+\mu_2}{c}(\tau+X)}H_2(\phi,\psi)(\tau)\d \tau,
\]
where $H_i(\phi,\psi)(i=1,2)$ are defined in (\ref{TruPro}).
For any $(\phi_i, \psi_i)\in\Gamma_X$, $i=1,2$, we have
\begin{align*}
&\ |\phi_1(\xi)f(\psi_1(\xi)) - \phi_2(\xi)f(\psi_2(\xi))|\\
\leq & \ |\phi_1(\xi)f(\psi_1(\xi)) - \phi_1(\xi)f(\psi_2(\xi))|+|\phi_1(\xi)f(\psi_2(\xi)) - \phi_2(\xi)f(\psi_2(\xi))|\\
\leq & \ S_0 f'(0) \max_{\xi\in[-X,X]}|\psi_1(\xi)-\psi_2(\xi)| + f'(0)e^{\lambda_1 X}\max_{\xi\in[-X,X]}|\phi_1(\xi)-\phi_2(\xi)|.
\end{align*}
Thus, it is easy to see that the operator $\mathcal{A}$ is continuous. Next, we show $\mathcal{A}$ is compact. Indeed, since $S_X$ and $I_X$ are class of $C^1([-X,X])$, note that
\begin{align*}
&\ \left|c(S_{X,1}'(\xi)-S_{X,2}'(\xi))+(2d_1+\mu_1)(S_{X,1}(\xi)-S_{X,2}(\xi))\right|\\
\leq &\ d_1|(\hat{\phi}_1(\xi+1)-\hat{\phi}_2(\xi+1))| + d_1|(\hat{\phi}_1(\xi-1)-\hat{\phi}_2(\xi-1))| + \beta|\phi_1(\xi)f(\psi_1(\xi)) - \phi_2(\xi)f(\psi_2(\xi))|\\
\leq &\ \beta S_0 f'(0) \max_{\xi\in[-X,X]}|\psi_1(\xi)-\psi_2(\xi)| + \left(2d_1+\beta f'(0)e^{\lambda_1 X}\right)\max_{\xi\in[-X,X]}|\phi_1(\xi)-\phi_2(\xi)|.
\end{align*}
Same arguments with $I_X'$, give us that $S_X'$ and $I_X'$ are bounded. Then the operator $\mathcal{A}$ is compact and is completely continuous. This ends the proof.
\end{proof}

Applying Schauder's fixed point theorem, we have the following lemma.
\begin{lem}
There exists $(S_X,I_X)\in\Gamma_X$ such that
\[
(S_X(\xi),I_X(\xi)) = \mathcal{A}(S_X,I_X)(\xi)
\]
for $\xi\in[-X,X]$.
\end{lem}

In the following, we show some prior estimates for $(S_X,I_X)$.
Define
\[
C^{1,1}([-X,X])=\{u\in C^1([-X,X])\ |\ u,u' \textrm{are Lipschitz continuous}\}
\]
with the norm
\begin{gather*}
  \|u\|_{C^{1,1}([-X,X])}=\max_{x\in[-X,X]}|u|+\max_{x\in[-X,X]}|u'|+
  \sup_{\begin{subarray}{c}  x,y\in [-X,X]   \\
                             x\neq y
        \end{subarray}}\frac{|u'(x)-u'(y)|}{|x-y|}.
\end{gather*}
\begin{lem}\label{LemC}
There exists a constant $C(Y)>0$ such that
\begin{equation*}
\|S_X\|_{C^{1,1}([-Y,Y])}\leq C(Y)\ \ {\rm and}\ \ \|I_X\|_{C^{1,1}([-Y,Y])}\leq C(Y)
\end{equation*}
for $0<Y<X$ and $X>-\mathfrak{X}_2$.
\end{lem}
\begin{proof}
Recall that $(S_X,I_X)$ is the fixed point of the operator $\mathcal{A},$ then
\begin{equation}
\label{FixEqu}\left\{
\begin{array}{l}
\vspace{2mm}
\displaystyle   cS_X'(\xi) = d_1\hat{S}_X(\xi+1) + d_1\hat{S}_X(\xi-1) - (2d_1+\mu_1)S_X(\xi) + \Lambda - \beta S_X(\xi)f(I_X(\xi)),\\
\displaystyle   cI_X'(\xi) = d_2\hat{I}_X(\xi+1) + d_2\hat{I}_X(\xi-1) - (2d_2+\mu_2)I_X(\xi) + \beta S_X(\xi) f(I_X(\xi)),
\end{array}\right.
\end{equation}
where
\begin{equation*}
\label{hatSV}
\hat{S}_X(\xi)=\left\{
\begin{array}{ll}
\displaystyle   S_X(X), &\mbox{for $\xi>X$,}
\\
\displaystyle   S_X(\xi), &\mbox{for $\xi\in[-X, X]$,}
\\
\displaystyle   S^-(\xi), &\mbox{for $\xi< -X$,}\\
\end{array}\right.\ \ \
\hat{I}_X(\xi)=\left\{
\begin{array}{ll}
\displaystyle   I_X(X), &\mbox{for $\xi>X$,}
\\
\displaystyle   I_X(\xi), &\mbox{for $\xi\in[-X, X]$,}
\\
\displaystyle   I^-(\xi), &\mbox{for $\xi< -X$.}\\
\end{array}\right.
\end{equation*}
Since $0\leq S_X(\xi)\leq S_0$ and $0\leq I_X(\xi)\leq e^{\lambda_1 Y}$ for all $\xi\in[-Y,Y]$,
from (\ref{FixEqu}) we have
\[
|S_X'(\xi)|\leq \frac{4d_1+\mu_1}{c}S_0 + \frac{\Lambda}{c} + \frac{\beta S_0f'(0)}{c}e^{\lambda_1 Y},
\]
and
\[
|I_X'(\xi)|\leq \frac{4d_2+\mu_2 + \beta S_0f'(0)}{c}e^{\lambda_1 Y}.
\]
Thus there exists some constant $C_1(Y) > 0$ such that
\[
\|S_X\|_{C^{1}([-Y,Y])}\leq C_1(Y)\ \ {\rm and}\ \ \|I_X\|_{C^{1}([-Y,Y])}\leq C_1(Y).
\]
For any $\xi,\eta\in[-Y,Y]$, it follows from \cite{ZhangWuIJB2019} that
\begin{align*}
&|\hat{S}_X(\xi+1)-\hat{S}_X(\eta+1)|\\ = &\left\{
\begin{array}{ll}
\displaystyle   |S_X(Y)-S_X(Y)| = 0, & \mbox{for $\xi+1,\eta+1>Y$,}
\\
\displaystyle   |S_X(\xi+1)-S_X(\eta+1)|\leq C_1(Y)|\xi-\eta|,   & \mbox{for $\xi+1,\eta+1<Y$,}
\\
\displaystyle   |S_X(\xi+1)-S_X(Y)|\leq C_1(Y)(Y-\xi-1)\leq C_1(Y)|\xi-\eta|,   & \mbox{for $\xi+1<Y,\eta+1>Y$,}
\\
\displaystyle   |S_X(X)-S_X(\eta+1)|\leq C_1(Y)(Y-\eta-1)\leq C_1(Y)|\xi-\eta|,   & \mbox{for $\xi+1>Y,\eta+1<Y$.}
\end{array}\right.
\end{align*}
Then $|\hat{S}_X(\xi+1)-\hat{S}_X(\eta+1)|\leq C_1(Y)|\xi-\eta|$ for all $\xi,\eta\in[-Y,Y]$. Similarly, we have
\[
|\hat{S}_X(\xi-1)-\hat{S}_X(\eta-1)|\leq C_1(Y)|\xi-\eta|
\]
for all $\xi,\eta\in[-Y,Y]$. Furthermore
\begin{align*}
&\ |\beta S_X(\xi)f(I_X(\xi)) - \beta S_X(\eta)f(I_X(\eta))|\\
\leq &\ |\beta S_X(\xi)f(I_X(\xi)) - \beta S_X(\xi)f(I_X(\eta))|+|\beta S_X(\xi)f(I_X(\eta)) - \beta S_X(\eta)f(I_X(\eta))|\\
\leq &\ \beta f'(0)C_1(Y)\left(|S_X(\xi)-S_X(\eta)| + |I_X(\xi)-I_X(\eta)|\right)
\end{align*}
for all $\xi,\eta\in[-Y,Y]$. Hence, there exist some constant $C(Y) > 0$ such that
\[
\|S_X\|_{C^{1,1}([-Y,Y])}\leq C(Y).
\]
Similarly,
\[
\|I_X\|_{C^{1,1}([-Y,Y])}\leq C(Y)
\]
for any $Y<X$. This completes the proof.
\end{proof}

\section{Existence of traveling wave solutions}
We first state the main results of this section as follows.
\begin{thm}\label{MainTh}
For any wave speed $c>c^*$, system \eqref{Model} admits a nontrivial traveling wave solution $(S(\xi),I(\xi))$ satisfying
\[
S^-\leq S(\xi)\leq S^+\ \ {\rm and}\ \  I^-\leq I(\xi)\leq I^+\ \ {\rm in}\ \ \mathbb{R}.
\]
Furthermore,
\[
\lim_{\xi\rightarrow-\infty}(S(\xi), I(\xi))=(S_0, 0)\ \ {\rm and}\ \ \lim_{\xi\rightarrow+\infty}(S(\xi), I(\xi))=(S^*,I^*).
\]
\end{thm}
The proof of Theorem \ref{MainTh} is divided into the following serval steps.

\textbf{Step 1.} We show that system \eqref{Model} admits a nontrivial traveling wave solution $(S(\xi),I(\xi))$ in $\mathbb{R}$ and satisfying $\lim_{\xi\rightarrow-\infty}(S(\xi), I(\xi))=(S_0, 0)$.

Choose $\{X_n\}_{n=1}^{+\infty}$ be an increasing sequence such that $X_n>-\mathfrak{X}_2$, $X_n>Y$ and $X_n\rightarrow+\infty$ as $n\rightarrow+\infty$ for all $n\in\mathbb{N}$, where $Y$ is from Lemma \ref{LemC}. Denote $(S_n, I_n)\in\Gamma_{X_n}$ be the solution of system (\ref{TruPro}). For any $N\in\mathbb{N}$, since the function $ I^+(\xi)$ is bounded in $[-X_N,X_N]$, then the sequences
\[
\{S_n\}_{n\geq N}\ \ {\rm and}\ \ \{ I_n\}_{n\geq N}
\]
are uniformly bounded in $[-X_N,X_N]$. Then by (\ref{TruPro}), we can obtain that
\[
\{S'_n\}_{n\geq N}\ \ {\rm and}\ \ \{ I'_n\}_{n\geq N}
\]
are also uniformly bounded in $[-X_N,X_N]$. Again with (\ref{TruPro}), we can express $S''_n(\xi)$ and $ I''_n(\xi)$ in terms of $S_n(\xi)$, $ I_n(\xi)$, $S_n(\xi\pm1)$, $ I_n(\xi\pm1)$, $S_n(\xi\pm2)$ and $ I_n(\xi\pm2)$, which give us
\[
\{S''_n\}_{n\geq N}\ \ {\rm and}\ \ \{ I''_n\}_{n\geq N}
\]
are uniformly bounded in $[-X_N+2,X_N-2]$. By the Arzela-Ascoli theorem (see \cite[Theorem A5]{Rudin1991}), we can use a diagonal process to extract a subsequence, denote by $\{S_{n_k}\}_{k\in\mathbb{N}}$ and $\{ I_{n_k}\}_{k\in\mathbb{N}}$ such that
\[
S_{n_k}\rightarrow S,\  I_{n_k}\rightarrow I,\ S'_{n_k}\rightarrow S'\  {\rm and}\  I'_{n_k}\rightarrow I'\ {\rm as}\ k\rightarrow+\infty
\]
uniformly in any compact subinterval of $\mathbb{R}$, for some functions $S$ and $ I$ in $C^1(\mathbb{R})$. Thus $(S(\xi), I(\xi))$ is a solution of system (\ref{WaveEqu}) with
\[
S^-(\xi)\leq S(\xi)\leq S^+(\xi)\ \ {\rm and}\ \  I^-(\xi)\leq I(\xi)\leq I^+(\xi)\ \ {\rm in}\ \ \mathbb{R}.
\]
Furthermore, by the definition of $S^-(\xi)$ and $I^-(\xi)$ from \eqref{UpLowSolution}, it follows that
\[
\lim_{\xi\rightarrow-\infty}(S(\xi), I(\xi))=(S_0, 0).
\]

\textbf{Step 2.}
We claim that the functions $S(\xi)$ and $I(\xi)$ satisfy $0<S(\xi)<S_0\ \ {\rm and}\ \  I(\xi)>0\ \ {\rm in}\ \ \mathbb{R}$.

We first show that $S(\xi)>0$ for all $\xi\in\mathbb{R}$. Assume reversely, that is assume that if there exists some real number $\xi_0$ such that $S(\xi_0) = 0$, then $S'(\xi_0) = 0$ and $J[S](\xi_0)\geq0$. By the first equation of (\ref{WaveEqu}), we have
\[
0 = d_1J[S](\xi_0) + \Lambda > 0,
\]
which is a contradiction. Thus $S(\xi)>0$ for all $\xi\in\mathbb{R}$.

Next, we show that $I(\xi)>0$ for all $\xi\in\mathbb{R}$. By way of contradiction, we assume that if there exists $\xi_1$ such that $I(\xi_1) = 0$ and $I(\xi)>0$ for all $\xi<\xi_1$.
From the second equation of (\ref{WaveEqu}), we have
\[
 I(\xi_1+1) +  I(\xi_1-1) = 0.
\]
Consequently, $ I(\xi_1+1) =  I(\xi_1-1) = 0$ since $I(\xi)\geq0$ in $\mathbb{R}$, which is a contradiction to the definition of $\xi_1$.

To show that $S(\xi)<S_0$ for all $\xi\in\mathbb{R}$, we assume that if there exists $\xi_2$ such that $S(\xi_2) = S_0$, it is easy to obtain
\[
0 = d_1J[S](\xi_2) - \beta S(\xi_2) f(I(\xi_2)) <0.
\]
This contradiction leads to $S(\xi)<S_0$ for all $\xi\in\mathbb{R}$.

\textbf{Step 3.} Boundedness of traveling wave solutions $S(\xi)$ and $I(\xi)$ in $\mathbb{R}$.

We need to consider two cases of the nonlinear incidence function $f(x)$.
In fact, the function $f(x)$ satisfying Assumption \ref{Assum} has two possibilities: (i) $\lim\limits_{x\rightarrow+\infty}f(x)$ exists; (ii)$\lim\limits_{x\rightarrow+\infty}f(x) = +\infty$. For example, the saturated incidence with $f(x)=\frac{bx}{1+cx}$ satisfy (i) since $\lim\limits_{x\rightarrow+\infty}\frac{bx}{1+cx} = \frac{b}{c}$ and the bilinear incidence with $f(x)=bx$ satisfy (ii).

\textbf{Case 1.} $\lim\limits_{x\rightarrow+\infty}f(x)$ exists. Without losing generality, we assume that $\lim\limits_{x\rightarrow+\infty}f(x) = \bar{f} < +\infty$, then it is easy to verify that $\frac{\Lambda}{\mu_1+\beta \bar{f}}$ is a lower solution of $S(\xi)$ and $\frac{\beta S_0\bar{f}}{\mu_2}$ is an upper solution of $I(\xi)$. Then we obtain
\begin{equation}\label{Bound}
\frac{\Lambda}{\mu_1+\beta \bar{f}}\leq S(\xi)<S_0\ \ {\rm and}\ \ 0< I(\xi)\leq\frac{\beta S_0 \bar{f}}{\mu_2}\ \ {\rm for\ \ all}\ \  \xi\in\mathbb{R}.
\end{equation}

\textbf{Case 2.} $\lim\limits_{x\rightarrow+\infty}f(x) = +\infty$. In this case, we have the following lemmas.

\begin{lem}\label{lem2}
The functions $\frac{ I(\xi\pm1)}{ I(\xi)}$ and $\frac{ I'(\xi)}{ I(\xi)}$ are bounded in $\mathbb{R}$.
\end{lem}
See Appendix A for the details of proof.
\begin{lem}\label{lem3}
Let $\{c_k,S_k, I_k\}$ be a sequence of traveling wave solutions of (\ref{Model}) with speed $\{c_k\}$ in a compact subinterval of $(0,\infty)$. If there is a sequence $\{\xi_k\}$ such that $ I(\xi_k)\rightarrow +\infty$ as $k\rightarrow +\infty$, then $S(\xi_k)\rightarrow 0$ as $k\rightarrow +\infty$.
\end{lem}
See Appendix B for the details of proof.
\begin{lem}\label{lem4}
If $\limsup\limits_{\xi\rightarrow+\infty} I(\xi)=+\infty$, then $\lim\limits_{\xi\rightarrow+\infty} I(\xi)=+\infty$.
\end{lem}
The proof of Lemma \ref{lem4} is similar to that of \cite[Lemma 3.4]{ChenGuoHamelNon2017}, so we omit the details.

\begin{lem}\label{lem5}
The function $ I(\xi)$ is bounded in $\mathbb{R}$.
\end{lem}
See Appendix C for the details of proof.

By the above lemmas, we know that $f(I(\xi))$ is bounded from above since $I(\xi)$ is bounded, then Proposition \eqref{Bound} follows.
Hence, we obtained that $S(\xi)$ and $I(\xi)$ are bounded from above
and $S(\xi)$ has a strictly positive lower bound in $\mathbb{R}$.\
In the following, we will show $I(\xi)$ could not approach zero.
\begin{lem}\label{lem6}
Let $0<c_1\leq c_2$ be given and $(S(\xi),I(\xi))$ be a solution of system \eqref{WaveEqu} with speed $c\in[c_1,c_2]$ satisfying
$0<S(\xi)<S_0\ \ {\rm and}\ \  I(\xi)>0\ \ {\rm in}\ \ \mathbb{R}$.
Then there exists some small enough constant $\varepsilon_0>0$, such that $I'(\xi)>0$ provided that $I(\xi)\leq\varepsilon_0$ for all $\xi\in \mathbb{R}$.
\end{lem}
See Appendix D for the details of proof.

\textbf{Step 4.} Convergence of the traveling wave solutions as $\xi\rightarrow+\infty$. The key point is to construct a suitable Lyapunov functional.

Let $g(x)=x-1-\ln x$, it is easy to check $g(x)\geq0$ since $g(x)$ has the global minimum value $0$ only at $x = 1$.
Define the following Lyapunov functional:
\[
L(S, I)(\xi) = W_1(S, I)(\xi) + d_1S^* W_2(S, I)(\xi) + d_2I^* W_3(S, I)(\xi),
\]
where
\[
W_1(S, I)(\xi) = c S^* g\left(\frac{S(\xi)}{S^*}\right) + c I^* g\left(\frac{I(\xi)}{I^*}\right),
\]
\[
W_2(S, I)(\xi) = \int_0^1 g\left(\frac{S(\xi-\theta)}{S^*}\right)\d \theta - \int_{-1}^0 g\left(\frac{S(\xi-\theta)}{S^*}\right)\d \theta
\]
and
\[
W_3(S, I)(\xi) = \int_0^1 g\left(\frac{ I(\xi-\theta)}{I^*}\right)\d \theta - \int_{-1}^0 g\left(\frac{ I(\xi-\theta)}{I^*}\right)\d \theta.
\]
Thanks to the boundedness of $S(\xi)$ and $I(\xi)$ (see Step 3), we have $W_1(S,I)(\xi)$ and $W_2(S,I)(\xi)$ are well defined and bounded from below.
Since $\lim_{\xi\rightarrow-\infty}I(\xi)=0$, we need to consider the process of $\xi$ approaching negative infinity for $W_3(S,I)(\xi)$. For the $\varepsilon_0$ in Lemma \ref{lem6}, define $\xi^* = \min\{\xi\in\mathbb{R}|I(\xi) = \varepsilon_0\}$, then $I(\xi)$ is increasing in $(-\infty, \xi^*]$. By the properties of function $g$, we have $W_3(S,I)(\xi)\geq0$ for $\xi\in(-\infty, \xi^*]$. Thus the Lyapunov function $L(S,I)(\xi)$ is well defined and bounded from below.

Next we show that the map $\xi\mapsto L(S, I)(\xi)$ is non-increasing. The derivative of $W_1(S, I)(\xi)$ along the solution of \eqref{WaveEqu} is calculated as follows
\begin{align*}
\frac{\d W_1(S, I)(\xi)}{\d \xi}\bigg|_{\eqref{WaveEqu}} = &\left(1-\frac{S^*}{S(\xi)}\right) c\frac{\d S(\xi)}{\d \xi} + \left(1-\frac{I^*}{ I(\xi)}\right) c\frac{\d  I(\xi)}{\d \xi}\\
= & \left(1-\frac{S^*}{S(\xi)}\right) \left(d_1J[S](\xi) + \Lambda - \mu_1 S(\xi) - \beta S(\xi)f( I(\xi))\right)\\
& + \left(1-\frac{I^*}{ I(\xi)}\right) \left(d_2J[ I](\xi) + \beta S(\xi)f(I(\xi)) - \mu_2  I(\xi)\right)\\
= & \left(1-\frac{S^*}{S(\xi)}\right) d_1J[S](\xi) + \left(1-\frac{I^*}{ I(\xi)}\right) d_2J[ I](\xi) + \Theta(\xi),
\end{align*}
where
\[
\Theta(\xi) = \left(1-\frac{S^*}{S(\xi)}\right) \left(\Lambda - \mu_1 S(\xi) - \beta S(\xi)f(I(\xi))\right) + \left(1-\frac{I^*}{ I(\xi)}\right) \left(\beta S(\xi)f(I(\xi)) - \mu_2  I(\xi)\right).
\]
Note that the endemic equilibrium $(S^*,I^*)$ of system (\ref{Model}) satisfying (\ref{EStar}) and $\mu_1 = \mu + \alpha$.
By some calculation, we obtain
\begin{align*}
\Theta(\xi)= &\  \mu_1 S^* \left(2-\frac{S^*}{S(\xi)}-\frac{S(\xi)}{S^*}\right) - \beta S^*f(I^*)\left[g\left(\frac{S^*}{S(\xi)}\right)+g\left(\frac{I^*S(\xi)f(I(\xi))}{I(\xi)S^*f(I^*)}\right)\right]\\
&-\beta S^* f(I^*)\left[g\left(\frac{I(\xi)}{I^*}\right) - g\left(\frac{f( I(\xi))}{f(I^*)}\right)\right].
\end{align*}
For $W_2(S, I)(\xi)$, one has that
\begin{align*}
\frac{\d W_2(S, I)(\xi)}{\d \xi}\bigg|_{\eqref{WaveEqu}} = &\ \frac{\d}{\d \xi} \left[\int_0^1 g\left(\frac{S(\xi-\theta)}{S^*}\right)\d \theta - \int_{-1}^0 g\left(\frac{S(\xi-\theta)}{S^*}\right)\d \theta\right]\\
= &\  \int_0^1 \frac{\d}{\d \xi}g\left(\frac{S(\xi-\theta)}{S^*}\right)\d \theta - \int_{-1}^0 \frac{\d}{\d \xi}g\left(\frac{S(\xi-\theta)}{S^*}\right)\d \theta\\
= &\ - \int_0^1 \frac{\d}{\d \theta}g\left(\frac{S(\xi-\theta)}{S^*}\right)\d \theta + \int_{-1}^0 \frac{\d}{\d \theta}g\left(\frac{S(\xi-\theta)}{S^*}\right)\d \theta\\
= &\ 2 g\left(\frac{S(\xi)}{S^*}\right) - g\left(\frac{S(\xi-1)}{S^*}\right) - g\left(\frac{S(\xi+1)}{S^*}\right).
\end{align*}
Similarly,
\[
\frac{\d W_3(S, I)(\xi)}{\d \xi}\bigg|_{\eqref{WaveEqu}} = 2 g\left(\frac{ I(\xi)}{I^*}\right) - g\left(\frac{ I(\xi-1)}{I^*}\right) - g\left(\frac{ I(\xi+1)}{I^*}\right).
\]
It can be shown that
\[
\left(1-\frac{S^*}{S(\xi)}\right) d_1J[S](\xi) + S^*\frac{\d W_2(S, I)(\xi)}{\d \xi}\bigg|_{\eqref{WaveEqu}} = - d_1S^* \left[g\left(\frac{S(\xi-1)}{S(\xi)}\right) + g\left(\frac{S(\xi+1)}{S(\xi)}\right)\right],
\]
and
\[
\left(1-\frac{I^*}{ I(\xi)}\right) d_2J[ I](\xi) + I^*\frac{\d W_3(S, I)(\xi)}{\d \xi}\bigg|_{\eqref{WaveEqu}} = - d_2I^* \left[g\left(\frac{ I(\xi-1)}{ I(\xi)}\right) + g\left(\frac{ I(\xi+1)}{ I(\xi)}\right)\right].
\]
Thus
\begin{align*}
\frac{\d L(S, I)(\xi)}{\d \xi}\bigg|_{\eqref{WaveEqu}}
= &\ - d_1S^* \left[g\left(\frac{S(\xi-1)}{S(\xi)}\right) + g\left(\frac{S(\xi+1)}{S(\xi)}\right)\right] - d_2I^* \left[g\left(\frac{ I(\xi-1)}{ I(\xi)}\right) + g\left(\frac{ I(\xi+1)}{ I(\xi)}\right)\right]\\
&\ -\beta_1 S^* f(I^*)\left[g\left(\frac{S^*}{S(\xi)}\right) + g\left(\frac{S(\xi)f(I(\xi))I^*}{S^*f(I^*) I(\xi)}\right) + g\left(\frac{ I(\xi)}{I^*}\right) - g\left(\frac{f( I(\xi))}{f(I^*)}\right)\right]\\
&\ +\mu S^* \left(2-\frac{S^*}{S(\xi)}-\frac{S(\xi)}{S^*}\right).
\end{align*}
Since the arithmetic mean of non-negative real numbers is greater than or equal to the geometric mean of the same list, then we have
\[
2 - \frac{S^*}{S(\xi)} - \frac{S(\xi)}{S^*}\leq 0.
\]
From Assumption \ref{Assum}, we can conclude that
\[
\left(1-\frac{f(I^*)}{f(I)}\right)\left(\frac{f(I)}{f(I^*)} - \frac{I}{I^*}\right)\leq 0.
\]
Then, we have
\begin{align*}
g\left(\frac{f( I(\xi))}{f(I^*)}\right) - g\left(\frac{ I(\xi)}{I^*}\right) =\ &\frac{f(I)}{f(I^*)} - \frac{I}{I^*} + \ln\left(\frac{I f(I^*)}{I^*f(I)}\right)\\
\leq\ &\frac{f(I)}{f(I^*)} - \frac{I}{I^*} + \frac{I f(I^*)}{I^*f(I)} - 1\\
=\ &\left(1-\frac{f(I^*)}{f(I)}\right)\left(\frac{f(I)}{f(I^*)} - \frac{I}{I^*}\right)\\
\leq\ &0.
\end{align*}
Here we use $\frac{I f(I^*)}{I^*f(I)} - 1 - \ln\left(\frac{I f(I^*)}{I^*f(I)}\right) \geq 0$. Hence, the map $\xi\mapsto L(S, I)(\xi)$ is non-increasing. Consider an increasing sequence $\{\xi_k\}_{k\geq 0}$ with $\xi_k>0$ such that $\xi_k\rightarrow+\infty$ when $k\rightarrow+\infty$ and denote
$$\{S_k(\xi)=S(\xi+\xi_k)\}_{k\geq 0}\ \ \textrm{and}\ \ \{ I_k(\xi)= I(\xi+\xi_k)\}_{k\geq 0}.$$
Since the functions $S$ and $I$ are bounded, the system (\ref{WaveEqu}) give us that the functions $S$ and $I$ have bounded derivatives.
Then by Arzela-Ascoli theorem, the functions $\{S_k(\xi)\}$ and $\{I_k(\xi)\}$ converge in $C_{loc}^{\infty}(\mathbb{R})$ as $k\rightarrow+\infty$, up to extraction of a subsequence,
one may assume that the sequences of $\{S_k(\xi)\}$ and $\{I_k(\xi)\}$ convergence towards some nonnegative $C^\infty$ functions $S_\infty$ and $ I_\infty.$ Furthermore, since $L(S, I)(\xi)$ is non-increasing on $\xi$ and bounded from below,
then there exists a constant $C_0$ and large $k$ such that
\[
C_0\leq L(S_k, I_k)(\xi)=L(S, I)(\xi+\xi_k)\leq L(S, I)(\xi).
\]
Therefore there exists some $\delta\in \mathbb{R}$ such that $\lim\limits_{k\rightarrow\infty} L(S_k, I_k)(\xi)=\delta$ for any $\xi\in \mathbb{R}$. By Lebesgue's dominated convergence theorem (see \cite[Theorem 11.32]{Rudin1976}), we have
\[
\lim_{k\rightarrow+\infty}L(S_k, I_k)(\xi)=L(S_\infty, I_\infty)(\xi),\ \xi\in \mathbb{R}.
\]
Thus
\[
L(S_\infty, I_\infty)(\xi)=\delta.
\]
Note that $\frac{\d L}{\d \xi}=0$ if and only if $S(\xi)\equiv S^*$ and $ I(\xi)\equiv I^*$, it follows that
\[
(S_\infty, I_\infty)\equiv (S^*,I^*).
\]

Hence, we complete the proof of Theorem \ref{MainTh}.

\begin{thm}\label{Thcstar}
For the wave speed $c=c^*$, system \eqref{Model} admits a nontrivial traveling wave solution $(S(\xi),I(\xi))$ satisfying
\[
S^-\leq S(\xi)\leq S^+\ \ {\rm and}\ \  I^-\leq I(\xi)\leq I^+\ \ {\rm in}\ \ \mathbb{R}.
\]
Furthermore,
\[
\lim_{\xi\rightarrow-\infty}(S(\xi), I(\xi))=(S_0, 0)\ \ {\rm and}\ \ \lim_{\xi\rightarrow+\infty}(S(\xi), I(\xi))=(S^*,I^*).
\]
\end{thm}

For the case $c=c^*$, we can obtain the existence of traveling wave solutions $(S(n+c^*t),I(n+c^*t))$ by an approximation technique used in \cite[Section 4]{ChenGuoHamelNon2017}. Since the Lyapunov functional is independent on $c$, we can also have that $(S(n+c^*t),I(n+c^*t))$ satisfies the asymptotic boundary conditions \eqref{Bound1} and \eqref{Bound2}.
So we omit the details.

\section{Nonexistence of traveling wave solutions}

In this section, we study the nonexistence of traveling wave solutions. Firstly, we show that $c>0$ if there exists a nontrivial positive solution $(S(\xi), I(\xi))$ of system (\ref{WaveEqu}) satisfying the asymptotic boundary conditions (\ref{Bound1}) and (\ref{Bound2}).

\begin{lem}\label{cpositive}
Assume that $\Re_0>1$ and there exists a nontrivial solution $(S(\xi), I(\xi))$ of system (\ref{Model}) satisfying the asymptotic boundary conditions (\ref{Bound1}) and (\ref{Bound2}). Then $c>0$, where $c$ is defined in (\ref{c}).
\end{lem}
\begin{proof}
Assume that $c\leq0$.
Since $S(\xi)\rightarrow S_0\ \ {\rm and}\ \ I(\xi)\rightarrow 0\ \ {\rm as}\ \ \xi\rightarrow-\infty$, there exists a $\xi^*<0$ such that
\begin{equation}\label{Equ5.1}
cI'(\xi)\geq d_2[I(\xi+1)+I(\xi-1)-2I(\xi)] + \frac{\beta S_0 f'(0)+\mu_2}{2f'(0)}(f'(0)-\epsilon)I(\xi) - \mu_2 I(\xi),
\end{equation}
here we used the condition $\Re_0>1$. Note that inequality (\ref{Equ5.1}) is valid for any $\epsilon\in(0,f'(0))$, then for $\xi<\xi^*$, we have
\begin{equation}\label{Equ5.2}
cI'(\xi)\geq d_2[I(\xi+1)+I(\xi-1)-2I(\xi)] + \frac{\beta S_0 f'(0)-\mu_2}{2}I(\xi).
\end{equation}
Denote $\omega = \frac{\beta S_0 f'(0)-\mu_2}{2}$ and $Q(\xi) = \int_{-\infty}^\xi I(y)\d y$ for $\xi\in \mathbb{R}$, note that $\omega>0$ since $\Re_0>1$.
Integrating inequality (\ref{Equ5.1}) from $-\infty$ to $\xi$ and using $I(-\infty)=0$, one has that
\begin{equation}\label{Equ5.3}
cI(\xi) \geq d_2[Q(\xi+1)+Q(\xi-1)-2Q(\xi)] + \omega Q(\xi)\ \ \textrm{for}\ \ \xi<\xi^*.
\end{equation}
Again, Integrating inequality (\ref{Equ5.1}) from $-\infty$ to $\xi$, yields
\begin{equation}\label{Equ5.4}
cQ(\xi) \geq d_2\left(\int_\xi^{\xi+1}Q(\tau)\d\tau - \int_{\xi-1}^{\xi}Q(\tau)\d\tau\right) + \omega \int_{-\infty}^\xi Q(\tau)\d\tau\ \ \textrm{for}\ \ \xi<\xi^*.
\end{equation}
Since $Q(\xi)$ is strictly increasing in $\mathbb{R}$ and $c\leq0$, we can conclude that
\[
0\geq c Q(\xi)\geq d_2\left(\int_\xi^{\xi+1}Q(\tau)\d\tau - \int_{\xi-1}^{\xi}Q(\tau)\d\tau\right) + \omega \int_{-\infty}^\xi Q(\tau)\d\tau>0,
\]
which is a contradiction. Hence $c>0$. The proof is finished.
\end{proof}

Now, we are in position to show the nonexistence of traveling wave solutions, we will use two-sided Laplace to prove it (see \cite{BaiZhangCNSNS2015,YangLiLiWangDCDSB2013,ZhouSongWeiJDE2019})
\begin{thm}\label{ThNon}
Assume that $\Re_0>1$ and $c<c^*$. Then there are no nontrivial solution $(S(\xi), I(\xi))$ of system (\ref{Model}) satisfying the asymptotic boundary conditions (\ref{Bound1}) and (\ref{Bound2}).
\end{thm}
\begin{proof}
By way of contradiction, assume that there exists a nontrivial positive solution $(S(\xi), I(\xi))$ of system (\ref{Model}) satisfying the asymptotic boundary condition (\ref{Bound1}) and (\ref{Bound2}). Then $c>0$ by Lemma \ref{cpositive} and
\[
S(\xi)\rightarrow S_0\ \ {\rm and}\ \ I(\xi)\rightarrow 0\ \ {\rm as}\ \ \xi\rightarrow-\infty.
\]
Let $\omega = \frac{\beta S_0 f'(0)-\mu_2}{2}$ and $Q(\xi) = \int_{-\infty}^\xi I(y)\d y$ for $\xi\in \mathbb{R}$. It follows from the proof of Lemma \ref{cpositive}, there exists a $\xi^*<0$, we have
\[
cQ(\xi) \geq d_2\left(\int_\xi^{\xi+1}Q(\tau)\d\tau - \int_{\xi-1}^{\xi}Q(\tau)\d\tau\right) + \omega \int_{-\infty}^\xi Q(\tau)\d\tau\ \ \textrm{for}\ \ \xi<\xi^*.
\]
Recalling that $Q(\xi)$ is strictly increasing in $\mathbb{R}$, one has that
\[
d_2\left(\int_\xi^{\xi+1}Q(\tau)\d\tau - \int_{\xi-1}^{\xi}Q(\tau)\d\tau\right)>0.
\]
Thus
\begin{equation}\label{Equ5.5}
cQ(\xi) \geq \omega \int_{-\infty}^\xi Q(\tau)\d\tau\ \ \textrm{for}\ \ \xi<\xi^*.
\end{equation}
Hence, there exists some constant $\delta>0$ such that
\begin{equation}\label{Equ5.6}
\omega \delta Q(\xi-\delta)\leq c Q(\xi)\ \ \textrm{for}\ \ \xi<\xi^*.
\end{equation}
Moreover, there exists a $\nu>0$ is large enough and $\epsilon_0\in(0,1)$ such that
\begin{equation}\label{Equ5.7}
Q(\xi-\nu)\leq \epsilon_0 Q(\xi)\ \ \textrm{for}\ \ \xi<\xi^*.
\end{equation}
Set
\[
\mu_0:=\frac{1}{\nu}\ln \frac{1}{\epsilon_0}\ \ \ \textrm{and}\ \ \ V(\xi):=Q(\xi)e^{-\mu_0\xi}.
\]
We have
\[
V(\xi-\nu) = Q(\xi-\nu)e^{-\mu_0(\xi-\nu)}<\epsilon_0 Q(\xi)e^{-\mu_0(\xi-\nu)} = V(\xi)\ \ \textrm{for}\ \ \xi<\xi^*,
\]
which implies that $V(\xi)$ is bounded as $\xi\rightarrow-\infty$.
Since $\int_{-\infty}^\infty I(\xi)\d\xi<\infty$, we obtain that
\[
\lim_{\xi\rightarrow\infty}V(\xi) = \lim_{\xi\rightarrow\infty} Q(\xi)e^{-\mu_0\xi} = 0.
\]
From the second equation of (\ref{WaveEqu}), we have
\[
cI'(\xi)\leq d_2[I(\xi+1)+I(\xi-1)-2I(\xi)] +\beta S_0 f'(0)I(\xi) -\mu_2 I(\xi),
\]
integrating over $(-\infty,\xi)$, give us
\[
cI(\xi)\leq d_2[Q(\xi+1)+Q(\xi-1)-2Q(\xi)] + \beta S_0 f'(0)Q(\xi) -\mu_2 Q(\xi).
\]
Hence, we can obtain that
\[
\sup_{\xi\in\mathbb{R}} \left\{I(\xi)e^{-\mu_0 \xi}\right\}<\infty\ \ \ \textrm{and}\ \ \ \sup_{\xi\in\mathbb{R}} \left\{I'(\xi)e^{-\mu_0 \xi}\right\}<\infty.
\]
For $\lambda\in\mathbb{C}$ with $0<\textrm{Re}\lambda<\mu_0$, define the following two-sided Laplace transform of $I(\cdot)$ by
\[
\mathcal{L}(\lambda):=\int_{-\infty}^\infty e^{-\lambda \xi}I(\xi)\d\xi.
\]
Note that
\begin{align*}
&\ \int_{-\infty}^\infty e^{-\lambda\xi}[I(\xi+1)+I(\xi-1)]\d\xi\\
=&\ e^\lambda \int_{-\infty}^\infty e^{-\lambda(\xi+1)}I(\xi+1)\d\xi + e^{-\lambda} \int_{-\infty}^\infty e^{-\lambda(\xi-1)}I(\xi-1)\d\xi\\
=&\ \left(e^\lambda+e^{-\lambda}\right) \mathcal{L}(\lambda)
\end{align*}
and
\begin{align*}
\int_{-\infty}^\infty e^{-\lambda\xi}I'(\xi)\d\xi
=e^\lambda I(\xi)\bigg|_{-\infty}^\infty - \int_{-\infty}^\infty I(\xi)\d e^{-\lambda\xi}
=\lambda \mathcal{L}(\lambda).
\end{align*}
From the second equation of (\ref{WaveEqu}), we have
\begin{equation}\label{Equ5.8}
d_2J[I](\xi) + \beta S_0 f'(0) I(\xi) - \mu_2 I(\xi) - cI'(\xi) =\beta S_0 f'(0) I(\xi) -\beta S(\xi)f(I(\xi)).
\end{equation}
Taking two-sided Laplace transform on (\ref{Equ5.8}), give us
\begin{equation}\label{Equ5.9}
\Delta(\lambda,c)\mathcal{L}(\lambda) =  \int_{-\infty}^\infty e^{-\lambda\xi} \left[\beta S_0 f'(0) I(\xi) -\beta S(\xi)f(I(\xi))\right]\d\xi.
\end{equation}
It follows from the proof in Lemma \ref{LemLowI}, as $\xi\rightarrow-\infty$, we have
\begin{align*}
\left[\beta S_0 f'(0) I(\xi) -\beta S(\xi)f(I(\xi))\right] e^{-2\mu_0\xi} \leq &\ I^2(\xi) e^{-2\mu_0\xi}\\
\leq &\ \left(\sup_{\xi\in\mathbb{R}}\left\{I(\xi)e^{-\mu_0\xi}\right\}\right)^2\\
\leq &\ \infty.
\end{align*}
Thus, we can obtain that
\begin{equation}\label{Equ5.10}
\sup_{\xi\in\mathbb{R}}\left[\beta S_0 f'(0) I(\xi) -\beta S(\xi)f(I(\xi))\right] e^{-2\mu_0\xi}<\infty.
\end{equation}
By the property of Laplace transform \cite{Widder1941}, either there exist a real number $\mu_0$ such that $\mathcal{L}(\lambda)$ is analytic for $\lambda\in\mathbb{C}$ with $0<\textrm{Re} \lambda < \mu_0$ and $\lambda=\mu_0$ is singular point of $\mathcal{L}(\lambda)$, or $\mathcal{L}(\lambda)$ is well defined for $\lambda\in\mathbb{C}$ with $\textrm{Re} \lambda>0$. Furthermore,
the two Laplace integrals can be analytically continued to the whole right half line;
otherwise, the integral on the left of (\ref{Equ5.9}) has singularity at $\lambda = \mu_0$ and it is analytic for all $\lambda<\mu_0$. However, it follows from (\ref{Equ5.10}) that the integral on the right of (\ref{Equ5.9}) is actually analytic for all $\lambda\leq 2 \mu_0$, a contradiction. Thus, (\ref{Equ5.9}) holds for all $\textrm{Re}\lambda>0$.
From Lemma \ref{WaveSpeed}, $\Delta(\lambda,c)>0$ for all $\lambda>0$ and by the definition of $\Delta(\lambda,c)$ in (\ref{Delta}), we know that $\Delta(\lambda,c)\rightarrow\infty$ as $\lambda\rightarrow\infty$, which is a contradiction with Equation (\ref{Equ5.9}).
This ends the proof.
\end{proof}

\section{Application and discussion}\label{sec:dis}
As an application, we consider the following two discrete diffusive epidemic models. The first one is a model with saturated incidence rate which has been wildly used in epidemic modeling (see, for example, \cite{LiLiYangAMC2014,ZhangWuIJB2019,xu2009global,zhang2018behavior,XuGuoMMAS2019}).

\begin{ex}
Discrete diffusive epidemic model with saturated incidence rate:
\begin{equation}
\label{Eg1}\left\{
\begin{array}{l}
\vspace{2mm}
\displaystyle   \frac{\d S_n(t)}{\d t}= d_1[S_{n+1}(t) + S_{n-1}(t) - 2S_n(t)] + \Lambda - \frac{\beta S_n(t)I_n(t)}{1+\alpha I_n(t)} - \mu_1 S_n(t),\\
\displaystyle   \frac{\d I_n(t)}{\d t}= d_2[I_{n+1}(t) + I_{n-1}(t) - 2I_n(t)] + \frac{\beta S_n(t)I_n(t)}{1+\alpha I_n(t)}  - \gamma I_n(t) - \mu_1 I_n(t),
\end{array}\right.
\end{equation}
where $\beta I_n(t)$ is the force of infection and $\frac{1}{1+\alpha I_n(t)}$ measures the inhibition effect which is dependent on the infected individuals.
\end{ex}

Setting $f(I_n(t)) = \frac{\beta S_n(t)I_n(t)}{1+\alpha I_n(t)}$ in the original system \eqref{Model}, we can easily see that \eqref{Eg1} is a special case of \eqref{Model}.
In fact, it is obvious that $f(I_n(t))$ satisfy Assumption \ref{Assum}. The disease-free equilibrium of system \eqref{Eg1} is similar to the original one, which is $\tilde{E}_{0}=(S_0,0)$. Moreover, we obtain the basic reproduction number of system \eqref{Eg1} as $\Re_{1} = \frac{\beta S_0}{\gamma+\mu_1}$ and there exists a positive equilibrium
$\tilde{E}^{*}=(\tilde{S}^{*},\tilde{I}^{*})$ if $\Re_{1}>1$, where
\[
\tilde{S}^{*} = \frac{\alpha\Lambda+\gamma+\mu_1}{\beta+\alpha\mu_1}\ \ \ {\rm and}\ \ \ \tilde{I}^{*} = \frac{\Lambda\beta - \mu_1(\gamma+\mu_1)}{(\gamma+\mu_1)(\beta+\alpha\mu_1)}.
\]

Hence, from Theorems \ref{MainTh}, \ref{Thcstar} and \ref{ThNon}, we obtain the following corollary
\begin{cor}\label{Cor1}
Assume that $\Re_1>1$. Then there exists some $c^*>0$ such that for any $c\geq c^*$, system \eqref{Eg1} admits a traveling wave solution $(S(\xi),I(\xi))$ satisfying
\begin{equation}\label{AC1}
\lim_{\xi\rightarrow-\infty}(S(\xi), I(\xi))=(S_0, 0)\ \ {\rm and}\ \ \lim_{\xi\rightarrow+\infty}(S(\xi), I(\xi))=(\tilde{S}^{*},\tilde{I}^{*}).
\end{equation}
Furthermore, system \eqref{Eg1} admits no traveling wave solutions satisfying \eqref{AC1} when $c<c^*$.
\end{cor}

The next example was studied in \cite{ChenGuoHamelNon2017}, and our results will solve the open problem proposed in \cite{ChenGuoHamelNon2017}, which is the traveling wave solutions converge to the endemic equilibrium as $\xi\rightarrow+\infty$ for discrete diffusive system (\ref{PreModel}).
\begin{ex}
Discrete diffusive epidemic model with mass action infection mechanism:
\begin{equation}
\label{Eg2}\left\{
\begin{array}{l}
\vspace{2mm}
\displaystyle   \frac{\d S_n(t)}{\d t}= d_1[S_{n+1}(t) + S_{n-1}(t) - 2S_n(t)] + \Lambda - \beta S_n(t)I_n(t) - \mu_1 S_n(t),\\
\displaystyle   \frac{\d I_n(t)}{\d t}= d_2[I_{n+1}(t) + I_{n-1}(t) - 2I_n(t)] + \beta S_n(t)I_n(t)  - \gamma I_n(t) - \mu_1 I_n(t).
\end{array}\right.
\end{equation}
\end{ex}

Setting $f(I_n(t)) = \beta S_n(t)I_n(t)$ in the original system \eqref{Model}, we can easily see that \eqref{Eg2} is a special case of \eqref{Model} and this model has been studied in \cite{ChenGuoHamelNon2017}.
The disease-free equilibrium of system \eqref{Eg2} is $\bar{E}_{0}=(S_0,0)$. Moreover, we obtain the basic reproduction number of system \eqref{Eg2} is the same with \eqref{Eg1} as $\Re_{1} = \frac{\beta S_0}{\gamma+\mu_1}$ and there exists a positive equilibrium
$\bar{E}^{*}=(\bar{S}^{*},\bar{I}^{*})$ if $\Re_{1}>1$, where
\[
\bar{S}^{*} = \frac{\gamma+\mu_1}{\beta}\ \ \ {\rm and}\ \ \ \bar{I}^{*} = \frac{\Lambda - \mu_1\bar{S}^*}{\beta\bar{S}^*}.
\]

Then, from Theorems \ref{MainTh}, \ref{Thcstar} and \ref{ThNon}, we obtain the following corollary
\begin{cor}\label{Cor2}
Assume that $\Re_1>1$. Then there exists some $c^*>0$ such that for any $c\geq c^*$, system \eqref{Eg2} admits a traveling wave solution $(S(\xi),I(\xi))$ satisfying
\begin{equation}\label{AC2}
\lim_{\xi\rightarrow-\infty}(S(\xi), I(\xi))=(S_0, 0)\ \ {\rm and}\ \ \lim_{\xi\rightarrow+\infty}(S(\xi), I(\xi))=(\bar{S}^{*},\bar{I}^{*}).
\end{equation}
Furthermore, system \eqref{Eg2} admits no traveling wave solutions satisfying \eqref{AC2} when $c<c^*$.
\end{cor}
Note that Corollary \ref{Cor2} could answer the open problem proposed in \cite{ChenGuoHamelNon2017}, that is the traveling wave solutions for system \eqref{PreModel3} converge to the endemic equilibrium at $+\infty$.

Next, we show that how the parameters affect wave speed. Suppose $(\hat{\lambda},\hat{c})$ be a zero root of $\Delta(\lambda,c)$ which defined in \eqref{Delta}, a direct calculation yields
\[
\frac{\d \hat{c}}{\d \beta} = \frac{S_0f'(0)}{\lambda}>0,\ \ \frac{\d \hat{c}}{\d d_2} = \frac{e^{\lambda}+e^{-\lambda}-2}{\lambda}>0\ \ {\rm and}\ \ \frac{\d \hat{c}}{\d \Re_0} = \frac{\mu_2}{\lambda}>0.
\]
that is, $\hat{c}$ is an increasing function on $\beta$, $d_2$ and $\Re_0$. Biologically, this means that the diffusion and infection ability of infected individuals can accelerate the speed of disease spreading.

Now, we are in a position to make the following summary:

In this paper, a discrete diffusive epidemic model with nonlinear incidence rate has been investigated.
When the basic reproduction number $\Re_0>1$, we proved that there exists a critical wave speed $c^*>0$, such that for each $c \geq c^*$ the system (\ref{Model}) admits a nontrivial traveling wave solution. Moreover, we used a Lyapunov functional to establish the convergence of traveling wave solutions at $+\infty$. We also showed the nonexistence nontrivial traveling wave solutions when $\Re_0>1$ and $c<c^*$. As special example of the model \eqref{Model}, we considered two different discrete diffusive epidemic model and apply our general results to show the conditions of existence and nonexistence of traveling wave solutions for the model \eqref{Eg1}.
One of the example is studied in \cite{ChenGuoHamelNon2017} and our result solved the open problem proposed in \cite{ChenGuoHamelNon2017}, which is the traveling wave solutions converge to the endemic equilibrium as $\xi\rightarrow+\infty$ for discrete diffusive system (\ref{PreModel}).

Here we mention some functions $f(I)$ considered
in the literature that do not satisfy Assumptions \ref{Assum}. For example, the incidence rates with media impact  $f(I) = I e^{-mI}$ in \cite{CuiSunZhuJDDE2008}; the specific incidence rate $f(I) = \frac{kI}{1 + \alpha I^2}$ in \cite{XiaoRuanMB2007}; and the nonmonotone incidence rate $f(I) = \frac{kI}{1 + \beta I + \alpha I^2}$ in \cite{XiaoZhouCAMQ2006}. In a recent paper, Shu et al. \cite{ShuPanWangWuJDDE2019} studied a SIR model with non-monotone incidence rates and without constant recruitment, they investigated the existence and nonexistence of traveling wave solutions. What is the condition of existence and nonexistence of traveling wave solution for our model (\ref{Model}) with non-monotone incidence rates, which will be an interesting question and we leave this for future work.



\section*{Appendix A: Proof of Lemma \ref{lem2}}
\begin{proof}
From the second equation of (\ref{WaveEqu}), one has that
\[
c  I'(\xi) + (2d_2 + \mu_2) I(\xi) = d_2 I(\xi+1) +  d_2 I(\xi-1) + \beta S(\xi) f( I(\xi)) >0.
\]
Denote $U(\xi) = e^{\nu \xi} I(\xi)$, where $\nu = (2d_2 + \mu_2)/c$. It follows that
\[
c U'(\xi) = e^{\nu \xi}(c I'(\xi) + (2d_2 + \mu_2) I(\xi))>0,
\]
thus $U(\xi)$ is increasing on $\xi$. Then $U(\xi-1) < U(\xi)$, that is
\[
\frac{ I(\xi-1)}{ I(\xi)} < e^\nu\ \ {\rm for}\ \ {\rm all}\ \ \xi\in\mathbb{R}.
\]
Note that
\begin{align}\label{Equ3}
\nonumber\left[e^{\nu\xi} I(\xi)\right]' & = \frac{1}{c}e^{\nu\xi}\left[d_2 I(\xi+1) + d_2 I(\xi-1) + \beta S(\xi) f(I(\xi))\right]\\
& > \frac{d_2}{c}e^{\nu\xi} I(\xi+1).
\end{align}
Integrating (\ref{Equ3}) over $[\xi,\xi+1]$ and using the fact that $e^{\nu \xi} I(\xi)$ is increasing, we have
\begin{align*}
e^{\nu(\xi+1)} I(\xi+1)\ > & \ e^{\nu\xi} I(\xi) + \frac{d_2}{c}\int_\xi^{\xi+1}e^{\nu s} I(s+1)\d s\\
> & \ e^{\nu\xi} I(\xi) + \frac{d_2}{c}\int_\xi^{\xi+1}e^{\nu (\xi+1)} I(\xi+1)e^{-\nu}\d s\\
= & \ e^{-\nu}\left[ I(\xi) + \frac{d_2}{c} I(\xi+1)\right].
\end{align*}
By (\ref{Equ3}), we obtain
\begin{equation}\label{Equ4}
\left[e^{\nu\xi} I(\xi)\right]' > \left(\frac{d_2}{c}\right)^2 e^{-2\nu}e^{\nu(\xi+1)} I(\xi+1).
\end{equation}
Integrating (\ref{Equ4}) over $[\xi-\frac{1}{2},\xi]$, yields
\begin{align*}
e^{\nu\xi} I(\xi) > & \left(\frac{d_2}{c}\right)^2 e^{-2\nu}\int_{\xi-\frac{1}{2}}^{\xi}e^{\nu(s+1)} I(s+1)\d s\\
> & \left(\frac{d_2}{c}\right)^2 \frac{e^{-2\nu}}{2} e^{\nu(\xi+\frac{1}{2})} I\left(\xi+\frac{1}{2}\right),
\end{align*}
that is
\[
\frac{ I\left(\xi+\frac{1}{2}\right)}{ I(\xi)} < 2 \left(\frac{c}{d_2}\right)^2 e^{\frac{3}{2}\nu}\ \ {\rm for}\ \ {\rm all}\ \ \xi\in\mathbb{R}.
\]
Similarly, integrating (\ref{Equ4}) over $[\xi, \xi+\frac{1}{2}]$, we have
\[
\frac{ I(\xi+1)}{ I\left(\xi+\frac{1}{2}\right)} < 2 \left(\frac{c}{d_2}\right)^2 e^{\frac{3}{2}\nu}\ \ {\rm for}\ \ {\rm all}\ \ \xi\in\mathbb{R}.
\]
Thus
\[
\frac{ I(\xi+1)}{ I(\xi)} = \frac{ I\left(\xi+\frac{1}{2}\right)}{ I(\xi)} \frac{ I(\xi+1)}{ I\left(\xi+\frac{1}{2}\right)} < 4 \left(\frac{c}{d_2}\right)^4 e^{3\nu}\ \ {\rm for}\ \ {\rm all}\ \ \xi\in\mathbb{R}.
\]
By the second equation of (\ref{WaveEqu}), it follows that
\begin{align*}
c \frac{I'(\xi)}{I(\xi)}  = & \frac{I(\xi+1)}{I(\xi)} + \frac{I(\xi-1)}{I(\xi)} + \beta S(\xi) \frac{f(I(\xi))}{I(\xi)} - (2d_2 + \mu_2)\\
\leq & \frac{I(\xi+1)}{I(\xi)} + \frac{ I(\xi-1)}{I(\xi)} + \beta S_0 f'(0) - (2d_2 + \mu_2),
\end{align*}
which gives us $\frac{ I'(\xi)}{ I(\xi)}$ is bounded in $\mathbb{R}$. The proof is end.
\end{proof}

\section*{Appendix B: Proof of Lemma \ref{lem3}}
\begin{proof}
Assume that there is a subsequence of $\{\xi_k\}_{k\in\mathbb{N}}$ again denoted by $\xi_k$, such that $ I_k(\xi_k)\rightarrow+\infty$ as $k\rightarrow+\infty$ and $S_k(\xi_k)\geq\varepsilon$ in $\mathbb{R}$ for all $k\in\mathbb{N}$ with some positive constant $\varepsilon$. From the first equation of (\ref{WaveEqu}), we have
\[
S'_k(\xi)\leq \frac{2S_0+\Lambda}{\tilde{c}} := \delta_0\ \ \textrm{in}\ \ \mathbb{R},
\]
where $\tilde{c}$ is a positive lower bound of $\{c_k\}$.
It follows that
\[
S_k(\xi)\geq\frac{\varepsilon}{2},\ \ \ \forall\xi\in[\xi_k-\delta,\xi_k]
\]
for all $k\in\mathbb{N}$, where $\delta = \frac{\varepsilon}{\delta_0}$.
By Lemma \ref{lem2}, we can assume that $\bigg|\frac{ I'_k(\xi)}{ I_k(\xi)}\bigg|<C_0$ for some $C_0>0$.
Then
\[
\frac{ I_k(\xi_k)}{ I_k(\xi)} = \exp\left\{\int_{\xi}^{\xi_k}\frac{ I'_k(s)}{ I_k(s)}\d s\right\}\leq e^{C_0\delta},\ \ \forall\xi\in[\xi_k-\delta, \xi_k]
\]
for all $k\in\mathbb{N}$. Thus
\[
\min_{\xi\in[\xi_k-\delta, \xi_k]} I_k(\xi)\geq e^{-C_0\delta} I_k(\xi_k),
\]
which give us
\[
\min_{\xi\in[\xi_k-\delta, \xi_k]} I_k(\xi) \rightarrow +\infty\ \ \textrm{as}\ \ k\rightarrow+\infty
\]
since $ I_k(\xi_k)\rightarrow+\infty$ as $k\rightarrow+\infty$.
Recalling we assumed that $\lim_{x\rightarrow+\infty}f(x)=+\infty$ in this case, one has that
\[
\max_{\xi\in[\xi_k-\delta, \xi_k]}S'_k(\xi)\leq \delta_0 - \frac{\beta\varepsilon}{2}f(\varpi_k)\rightarrow-\infty\ \ \textrm{as}\ \ k\rightarrow+\infty.
\]
where $\varpi_k:=\min\limits_{\xi\in[\xi_k-\delta, \xi_k]} I_k(\xi)$. Moreover, there exists some $K>0$ such that
\[
S'_k(\xi)\leq - \frac{2S_0}{\delta},\ \ \forall k\geq K\ \ \textrm{and}\ \ \xi\in[\xi_k-\delta, \xi_k].
\]
Note that $S_k<S_0$ in $\mathbb{R}$ for each $k\in\mathbb{N}$. Hence $S_k(\xi_k)\leq-S_0$ for all $k\geq K$, which reduces to a contradiction since $S_k(\xi_k)\geq\varepsilon$ in $\mathbb{R}$ for all $k\in\mathbb{N}$ with some positive constant $\varepsilon$. This completes the proof.
\end{proof}

\section*{Appendix C: Proof of Lemma \ref{lem5}}
\begin{proof}
Assume that $\limsup\limits_{\xi\rightarrow+\infty} I(\xi)=+\infty$, then we have $\lim\limits_{\xi\rightarrow+\infty}S(\xi)=0$ by Lemma \ref{lem3} and Lemma \ref{lem4}.
Set $\theta(\xi)=\frac{ I'(\xi)}{ I(\xi)}$, from the second equation of (\ref{WaveEqu}), we have
\[
c\theta(\xi) = d_2 e^{\int_{\xi}^{\xi+1}\theta(s)\d s} + d_2 e^{\int_{\xi}^{\xi-1}\theta(s)\d s} - (2d_2 + \mu_2) + B(\xi),
\]
where
\[
B(\xi) = \beta S(\xi)\frac{f( I(\xi))}{ I(\xi)}.
\]
It is easy to have that $\lim\limits_{\xi\rightarrow+\infty}B(\xi) = 0$ since $\frac{f( I(\xi))}{I(\xi)}\leq f'(0)$ and $\lim\limits_{\xi\rightarrow+\infty}S(\xi)=0$. By using \cite[Lemma 3.4]{ChenGuoMA2003}, $\theta(\xi)$ has a finite limit $\omega$ at $+\infty$ and satisfies the following equation
\[
h(\omega,c) := d_2\left(e^\omega + e^{-\omega} - 2\right) -c\omega - \mu_2 = 0.
\]
By some calculations, we obtain
\[
h(0,c)<0,\ \ \frac{\partial h(\omega,c)}{\partial \omega}\bigg|_{\omega=0}<0,\ \ \frac{\partial^2 h(\omega,c)}{\partial \omega^2}>0\ \ \textrm{and}\ \ \lim_{\omega\rightarrow+\infty}h(\omega,c)=-\infty.
\]
Thus, there exists a unique positive real root $\omega_0$ of $h(\omega,c) = 0$.
Recall that $\lambda_1$ is the smaller real root of (\ref{Delta}) and $\lambda_2$ is the larger real root of (\ref{Delta}).
From Lemma \ref{WaveSpeed}, one has
\[
d_2\left(e^{\lambda_2} + e^{-\lambda_2} - 2\right) -c\lambda_2 - \mu_2 = -\beta S_0 f'(0)<0,
\]
thus, we have $\lambda_2<\omega_0$. Since $\lim\limits_{\xi\rightarrow+\infty}\theta(\xi) = \omega_0$, then there exists some $\tilde{\xi}$ large enough such that
\[
 I(\xi)\geq C \exp\left\{\left(\frac{\lambda_2+\omega_0}{2}\right)\xi\right\}\ \ {\rm for}\ \ {\rm all}\ \ \xi\geq\tilde{\xi}
\]
with some constant $C$. This is a contradiction since $ I(\xi)\leq e^{\lambda_1\xi}$ in $\mathbb{R}$ and $\lambda_1<\omega_0$.
This ends the proof.
\end{proof}

\section*{Appendix D: Proof of Lemma \ref{lem6}}
\begin{proof}
Assume by way of contradiction that there is no such $\varepsilon_0$, that is there exist some sequence $\{\xi_k\}_{k\in\mathbb{N}}$ with speed $c_k\in(\underline{c},\overline{c})$ such that $ I(\xi_k)\rightarrow0$ as $k\rightarrow+\infty$ and $I'(\xi_k)\leq0$, where $\underline{c}$ and $\overline{c}$ are two given positive real number. Denote
\[
S_k(\xi):= S(\xi_k+\xi)\ \  \textrm{and}\ \  I_k(\xi):=  I(\xi_k+\xi).
\]
Thus we have $ I_k(0)\rightarrow0$ as $k\rightarrow+\infty$ and $ I_k(\xi)\rightarrow0$ locally uniformly in $\mathbb{R}$ as $k\rightarrow+\infty$.
As a consequence, there also holds that $ I_k'(\xi)\rightarrow0$ locally uniformly in $\mathbb{R}$ as $k\rightarrow+\infty$ by the second equation of (\ref{WaveEqu}).
From the proof of \cite[Lemma 3.8]{ChenGuoHamelNon2017}, we can obtain that $S_\infty = S_0$.
Let $\Psi_k(\xi):=\frac{ I_k(\xi)}{ I_k(0)}$. In the view of
\[
\Psi_k'(\xi) = \frac{ I_k'(\xi)}{ I_k(0)} = \frac{ I_k'(\xi)}{ I_k(\xi)}\Psi_k(\xi),
\]
we have that $\Psi_k(\xi)$ and $\Psi_k'(\xi)$ are also locally uniformly in $\mathbb{R}$ as $k\rightarrow+\infty$. Letting $k\rightarrow+\infty$, thus
\[
c_\infty\Psi_\infty'(\xi)= d_2 J[\Psi_\infty](\xi) + \beta S_0 f(\Psi_\infty(\xi)) - \mu_2 \Psi_\infty(\xi).
\]
We claim that $\Psi_\infty(\xi)>0$ in $\mathbb{R}$. In fact, if there exists some $\xi_0$ such that $\Psi_\infty(\xi_0)=0$, we also have ${\Psi}'_\infty(\xi_0)=0$ since $\Psi_\infty(\xi)\geq0$,
then
\[
0 = d_2(\Psi_\infty(\xi_0+1)+\Psi_\infty(\xi_0-1)).
\]
Thus $\Psi_\infty(\xi_0+1) = \Psi_\infty(\xi_0-1) = 0$, it follows that $\Psi_\infty(\xi_0+m) = 0$ for all $m\in\mathbb{Z}$. Recall that
$c_\infty\Psi_\infty'(\xi) \geq  - \mu_2 \Psi_\infty(\xi)$, then the map $\xi\mapsto \Psi_\infty(\xi) e^{\frac{\mu_2\xi}{c_\infty}}$ is nondecreasing. Since it vanishes at $\xi_0+m$ for all $m\in\mathbb{Z}$ and $e^{\frac{\mu_2\xi}{c_\infty}}$ is increasing, one can concluded that $\Psi_\infty = 0$ in $\mathbb{R}$, which is a contradicts with $\Psi_\infty(0) = 1$.

Denote $\mathcal{Z}(\xi):=\frac{\Psi_\infty'(\xi)}{\Psi_\infty(\xi)}$, it is easy to verify that $\mathcal{Z}(\xi)$ satisfies
\begin{equation}\label{Z}
c_\infty \mathcal{Z}(\xi)= d_2 e^{\int^{\xi+1}_\xi \mathcal{Z}(s){\rm d} s}\d y + d_2 e^{\int^{\xi-1}_\xi \mathcal{Z}(s){\rm d} s}\d y - 2d_2 + \beta S_0 f'(0) - \mu_2,
\emph{}\end{equation}
here we use $ I_k\rightarrow0$ and $\frac{f( I_k)}{ I_k}\rightarrow f'(0)$ as $k\rightarrow+\infty$. Recalling \cite[Lemma 3.4]{ChenGuoMA2003}, $\mathcal{Z}(\xi)$ has finite limits $\omega_{\pm}$ as $\xi\rightarrow\pm\infty$, where $\omega_\pm$ are roots of
\[
c_\infty \omega_\pm = d_2\left(e^{\omega_\pm} + e^{-\omega_\pm} -2\right) + \beta S_0 f'(0) - \mu_2.
\]
By the analogous arguments in Lemma \ref{WaveSpeed}, we have $\omega_\pm>0$. Thus $\Psi_\infty'(\xi)$ is positive at $\pm\infty$ by the definition of $\mathcal{Z}(\xi)$.
Moreover, $\Psi_\infty'(\xi)>0$ for all $\xi\in\mathbb{R}$. Indeed, if there exists some $\xi^*$ such that $\mathcal{Z}(\xi^*) = \inf_{\mathbb{R}}\mathcal{Z}(\xi)$,
then $\mathcal{Z}(\xi^*) = 0$. Differentiating (\ref{Z}) gives us
\[
c_\infty \mathcal{Z}'(\xi) = d_2(\mathcal{Z}(\xi+1) - \mathcal{Z}(\xi))\frac{\Psi_\infty(\xi+1)}{\Psi_\infty(\xi)} + d_2(\mathcal{Z}(\xi-1) - \mathcal{Z}(\xi))\frac{\Psi_\infty(\xi-1)}{\Psi_\infty(\xi)},
\]
it follows that
\[
\mathcal{Z}(\xi^*) = \mathcal{Z}(\xi^*+1) = \mathcal{Z}(\xi^*-1).
\]
Hence $\mathcal{Z}(\xi^*) = \mathcal{Z}(\xi^*+m)$ for all $m\in\mathbb{Z}$. Then, there is
\[
\inf_{\mathbb{R}}\mathcal{Z}(\xi)\geq\min\{\mathcal{Z}(+\infty),\mathcal{Z}(-\infty)\}>0.
\]
So $\Psi_\infty'(\xi)>0$. From the definition of $\Psi_\infty(\xi)$, we have
\[
0<\Psi_\infty'(0) = \lim_{k\rightarrow+\infty}\Psi_k'(0) = \lim_{k\rightarrow+\infty}\frac{ I_k'(0)}{ I_k(0)}.
\]
Thus, $ I'(\xi_k) =  I_k'(0)>0$, which is a contradiction. This completes the proof.
\end{proof}

\end{document}